\documentclass[11pt,%                       
               a4paper,%                    
               oneside,%                    
               reqno% 								
               ]{amsart}
                                   
\usepackage[T1]{fontenc}

\usepackage{lmodern}
\usepackage[utf8]{inputenc}	
\usepackage[english]{babel}

\usepackage[dvipsnames]{xcolor}
\usepackage{fullpage}

\usepackage{amssymb}

\usepackage{slashed}

\usepackage{tikz}
\usetikzlibrary{cd,decorations.pathmorphing,positioning,arrows,matrix,calc,backgrounds,decorations.markings}

\usepackage{hyperref}

\newcommand{\numberset}{\mathbb} 
\newcommand{\N}{\numberset{N}} 
\newcommand{\Z}{\numberset{Z}} 
\newcommand{\Q}{\numberset{Q}}
\newcommand{\R}{\numberset{R}}

\newcommand{\bA}{\mathbb A}
\newcommand{\bC}{\numberset{C}}
\newcommand{\bH}{\mathbb{H}}
\newcommand{\bT}{\mathbb T}
\newcommand{\cA}{\mathcal A}

\newcommand{\cD}{\mathcal D}

\newcommand{\cJ}{\mathcal J}
\newcommand{\cL}{\mathcal L}
\newcommand{\cO}{\mathcal O}

\newcommand{\cS}{\mathcal S}
\newcommand{\cT}{\mathcal T}

\newcommand{\cX}{\mathcal X}

\newcommand{\fra}{\mathfrak a}
\newcommand{\frp}{\mathfrak p}

\DeclareMathOperator{\KK}{KK}

\DeclareMathOperator{\Tot}{Tot}
\DeclareMathOperator{\Tor}{Tor}
\DeclareMathOperator{\dom}{dom}
\DeclareMathOperator{\Tr}{Tr}

\newcommand{\medwedge}{\mathop{\textstyle\bigwedge}\nolimits}

\newcommand{\sfZ}{\underline{\mathbb Z}}
\newcommand{\sfo}{\underline{o}}

\newcommand{\absv}[1]{\left|#1\right|}

\newcommand{\Gal}{\mathrm{Gal}}
\newcommand{\ab}{\mathrm{ab}}
\newcommand{\id}{\mathrm{id}}
\newcommand{\SL}{\mathrm{SL}}

\usepackage[hyperpageref]{backref}
\usepackage[nobysame,alphabetic,initials]{amsrefs}

\DefineSimpleKey{bib}{how}
\DefineSimpleKey{bib}{mrclass}
\DefineSimpleKey{bib}{mrnumber}
\DefineSimpleKey{bib}{fjournal}
\DefineSimpleKey{bib}{mrreviewer}

\renewcommand{\PrintDOI}[1]{%
  \href{http://dx.doi.org/#1}{{\tt DOI:#1}}%
}
\renewcommand{\eprint}[1]{#1}

\BibSpec{book}{%
    +{}  {\PrintPrimary}                {transition}
    +{.} { \PrintDate}                  {date}
    +{.} { \textit}                     {title}
    +{.} { }                            {part}
    +{:} { \textit}                     {subtitle}
    +{,} { \PrintEdition}               {edition}
    +{}  { \PrintEditorsB}              {editor}
    +{,} { \PrintTranslatorsC}          {translator}
    +{,} { \PrintContributions}         {contribution}
    +{,} { }                            {series}
    +{,} { \voltext}                    {volume}
    +{,} { }                            {publisher}
    +{,} { }                            {organization}
    +{,} { }                            {address}
    +{,} { }                            {status}
    +{,} { \PrintDOI}                   {doi}
    +{,} { \PrintISBNs}                 {isbn}
    +{}  { \parenthesize}               {language}
    +{}  { \PrintTranslation}           {translation}
    +{;} { \PrintReprint}               {reprint}
    +{.} { }                            {note}
    +{.} {}                             {transition}
}
\BibSpec{article}{%
    +{}  {\PrintAuthors}                {author}
    +{,} { \textit}                     {title}
    +{.} { }                            {part}
    +{:} { \textit}                     {subtitle}
    +{,} { \PrintContributions}         {contribution}
    +{.} { \PrintPartials}              {partial}
    +{,} { }                            {journal}
    +{}  { \textbf}                     {volume}
    +{}  { \PrintDatePV}                {date}
    +{,} { \issuetext}                  {number}
    +{,} { \eprintpages}                {pages}
    +{,} { }                            {status}
    +{,} { \PrintDOI}                   {doi}
    +{,} { \eprint}        {eprint}
    +{}  { \parenthesize}               {language}
    +{}  { \PrintTranslation}           {translation}
    +{;} { \PrintReprint}               {reprint}
    +{.} { }                            {note}
    +{.} {}                             {transition}
}
\BibSpec{collection.article}{%
    +{}  {\PrintAuthors}                {author}
    +{,} { \textit}                     {title}
    +{.} { }                            {part}
    +{:} { \textit}                     {subtitle}
    +{,} { \PrintContributions}         {contribution}
    +{,} { \PrintConference}            {conference}
    +{}  {\PrintBook}                   {book}
    +{,} { }                            {booktitle}
    +{,} { \PrintDateB}                 {date}
    +{,} { pp.~}                        {pages}
    +{,} { }                            {publisher}
    +{,} { }                            {organization}
    +{,} { }                            {address}
    +{,} { }                            {status}
    +{,} { \PrintDOI}                   {doi}
    +{,} { \eprint}        {eprint}
    +{}  { \parenthesize}               {language}
    +{}  { \PrintTranslation}           {translation}
    +{;} { \PrintReprint}               {reprint}
    +{.} { }                            {note}
    +{.} {}                             {transition}
}
\BibSpec{misc}{%
  +{}{\PrintAuthors}  {author}
  +{,}{ \textit}      {title}
  +{.}{ }             {how}
  +{}{ \parenthesize} {date}
  +{,} { available at \eprint}        {eprint}
  +{,}{ available at \url}{url}
  +{,}{ }             {note}
  +{.}{}              {transition}
}

\usepackage{amsthm}
\theoremstyle{plain}
\newtheorem{theorem}{Theorem}[section]
\newtheorem{proposition}[theorem]{Proposition}
\newtheorem{lemma}[theorem]{Lemma}
\newtheorem{corollary}[theorem]{Corollary}
\newtheorem*{corollary*}{Corollary}

\newtheorem{theoremA}{Theorem}

\theoremstyle{definition}

\theoremstyle{remark}
\newtheorem{example}[theorem]{Example}
\newtheorem{remark}[theorem]{Remark}

\usepackage[shortcuts]{extdash} % for hyphenating quasi-isomorphism

\author{Valerio Proietti}
\address{Department of Mathematics, University of Oslo, P.O. box 1053, Blindern, 0316 Oslo, Norway}
\email{valeriop@math.uio.no}

\author{Makoto Yamashita}
\address{Department of Mathematics, University of Oslo, P.O. box 1053, Blindern, 0316 Oslo, Norway}
\email{makotoy@math.uio.no}

\title[Structural results on groupoid homology]{Homology and K-theory of dynamical systems\\ IV. Further structural results on groupoid homology}

\date{v2: April 2, 2024, minor revision; v1: October 15, 2023}

\begin{document}
%******************************************************************
% Beginning
%******************************************************************

\begin{abstract}
We consider the homology theory of étale groupoids introduced by Crainic and Moerdijk, with particular interest to groupoids arising from topological dynamical systems.
We prove a Künneth formula for products of groupoids and a Poincaré-duality type result for principal groupoids whose orbits are copies of a Euclidean space.
We conclude with a few example computations for systems associated to nilpotent groups such as self-similar actions, and we generalize previous homological calculations by Burke and Putnam for systems which are analogues of solenoids arising from algebraic numbers.
For the latter systems, we prove the HK conjecture, even when the resulting groupoid is not ample.
\end{abstract}

\makeatletter
\@namedef{subjclassname@2020}{\textup{2020} Mathematics Subject Classification}
\makeatother

\subjclass[2020]{37B02; 22A22, 18G10}
\keywords{groupoid homology, Künneth formula, Poincaré duality, topological dynamics, derived functors.}

\maketitle
\setcounter{tocdepth}{1}
\tableofcontents

%******************************************************************
% Main
%******************************************************************
\addtocontents{toc}{\setcounter{tocdepth}{-10}}
\section*{Introduction}
\addtocontents{toc}{\setcounter{tocdepth}{1}}
In this paper we prove two structural results for the homology groups of topological dynamical systems, continuing our work \citelist{\cite{valmak:groupoid}\cite{valmak:groupoidtwo}}.
One is a Künneth type formula for the product of two systems, while the other reduces the homological computation to the (compactly supported) cohomology of the underlying space under certain conditions.

Previously we have focused on totally disconnected systems, which are represented by \emph{ample groupoids}, i.e., étale groupoids whose unit space is totally disconnected.
In this paper we do not work under this restriction but we deal with (more general) dynamical systems of \emph{finite topological dimension}.
A motivating class of examples is that of \emph{Smale spaces} \cite{ruelle:thermo}, which capture hyperbolicity on compact metric spaces, such as the dynamics of Anosov diffeomorphisms and more generally those on the basic sets of Axiom A diffeomorphisms.
 This has led to interesting intersection between the theory of dynamical systems and operator algebras.

We mainly work within the homology theory introduced by M.~Crainic and I.~Moerdijk \cite{cramo:hom}.
It defines homology groups with coefficient in equivariant sheaves to any étale groupoid, based on sheaves, derived formalism, and simplicial methods.

In another direction, given a locally compact étale groupoid (and more generally a locally compact groupoid with a continuous Haar system), there is a convolution product on the space of compactly supported continuous functions on the groupoid, which can be completed to a C$^*$-algebra \cite{MR584266}.
The $K$-groups of this C$^*$-algebra can be regarded as a homological invariants of the original groupoid which has better connection to index theory, classification of C$^*$-algebras, and other topics involving operator $K$-theory.

Comparison between these two theories has been a major driving force behind our previous works, and it plays an important role in this note too.

\medskip
Let us summarize the main conceptual results in this note.

\begin{theoremA}[Theorem \ref{thm:kunneth-formula}]\label{thmA:kunneth}
Let $G$ and $H$ be étale groupoids, and $\cS$ and $\cT$ be $G$- and $H$-equivariant sheaves.
Then there is a split short exact sequence
\begin{multline*}
0 \to {\bigoplus_{a + b = k}} H_a(G, \cS) \otimes H_b(H, \cT) \to H_k(G \times H, \cS\boxtimes\cT)\\
\quad \to {\bigoplus_{a + b = k-1}} \Tor(H_a(G, \cS), H_b(H, \cT)) \to 0.
\end{multline*}
\end{theoremA}

\begin{theoremA}[Theorem \ref{thm:poincare-duality}]\label{thmA:poincare-duality}
Let $\tilde G$ be a locally compact principal groupoid such that $\tilde G^x$ is homeomorphic to $\R^n$ for some fixed $n$.
Suppose there exists a generalized transversal $T$ in $\tilde G^{(0)}$, so that $G = \tilde G|_T$ is an étale groupoid.
We then have an isomorphism
\[
H_k(G, \sfZ) \cong H_c^{n-k}(\tilde G^{(0)}, \sfZ \times_{\Z/2\Z} \sfo),
\]
where on the right-hand side cohomology is computed with coefficients in the orbit-wise orientation sheaf over the unit space of $\tilde G$.
\end{theoremA}

In operator $K$-theory, Theorem \ref{thmA:kunneth} has a direct analogue which is the Künneth formula of $K$-groups of tensor product C$^*$-algebras due to Schochet \cite{MR650021}.
Theorem \ref{thmA:poincare-duality} can also be also interpreted as an analogue of Connes's Thom isomorphism for $K$-groups \cite{MR605351} (for a groupoid of the form $\R^n\ltimes X$, this result reduces operator $K$-groups of the associated $C^*$-crossed product to the topological $K$-groups of $X$, up to a degree shift determined by $n$).

However, the proofs are completely different from the $K$-theoretic ones.
While the proof of Theorem \ref{thmA:kunneth} is a standard Eilenberg--Zilber type manipulation of multicomplexes, for Theorem \ref{thmA:poincare-duality} we make an essential use of the sheaf theoretic idea and the formalism of derived functors.

Theorem \ref{thmA:poincare-duality} is particularly useful in computing the homology of groupoids coming from dynamical systems on manifolds and related structures.
Moreover, the idea behind this theorem is also useful to study systems of number theoretic origin.

We look at a class of Smale system $(Y^{(c)},\phi)$ that appears in the theory of algebraic actions due to Schmidt \cite{MR1345152}.
Given an algebraic number $c$, the space $Y^{(c)}$ is given as the Pontryagin dual of the additive group of a certain subring of $K=\Q(c)$, and $\phi$ is the natural map induced by the multiplication by $c$ (see Section \ref{sec:ntsol} for details on this construction).

\begin{theoremA}[Theorem \ref{thm:solstat}]\label{thmA:hk}
Denote by $G$ the (un)stable étale groupoid associated to $(Y^{(c)},\phi)$.
Then the groupoid homology of $G$ can be presented as the inductive limit of exterior power of a certain subgroup $\Gamma$ of the additive group of $K$,
\[
H_k(G,\sfZ)=\varinjlim \medwedge^{k+d} \Gamma,
\]
with respect to the connecting map $\theta_k \colon \medwedge^k \Gamma \to \medwedge^k \Gamma$. The shift is given by a natural number $d$ which depends on the infinite places of $K$. The map $\theta_k$ is the unique extension of $N\bigwedge^k m_{c^{-1}}$, with a natural number $N$, and $m_{c^{-1}}$ denoting multiplication by $c^{-1}$.
\end{theoremA}

An similar computation can be carried out for the K-groups of the groupoid C$^*$-algebra $C^*_rG$.

\begin{theoremA}[Theorem \ref{thm:hksolthm}]\label{thmA:hkk}
Denote by $G$ the (un)stable étale groupoid associated to $(Y^{(c)},\phi)$. With notation as in Theorem \ref{thmA:hk},
there is an isomorphism 
\[
K_i(C^*_rG)\cong \varinjlim \bigoplus_{k\in\Z}\medwedge^{i+d+2k} \Gamma.
\]
\end{theoremA}

Consequently, the $K$-group is isomorphic to the direct sum of groupoid homology with the same degree parity.
\begin{corollary*}[Corollary \ref{cor:HK-conj-for-num-theretic-solenoid}]
In the setting of Theorem \ref{thmA:hk} and Theorem \ref{thmA:hkk}, we have
\[
K_i(C^*_r G)\cong \bigoplus_{k\in \Z}H_{i + 2 k}(G,\sfZ).
\]
\end{corollary*}
This confirms, for the class of groupoids from Theorem \ref{thmA:hk}, the ``HK conjecture'' formulated by Matui \cites{MR3552533,MR3837599} in the setting of ample groupoids, which asks if there is an isomorphism between the $K$-groups and periodicized homology as above.

Theorem \ref{thmA:hk} is interesting for a few reasons.
Firstly, the computations given here generalize previous work by Burke and Putnam \cite{buput:ntsol}.
Secondly, the groupoids appearing in this context are not necessarily ample, providing us with examples where the HK conjecture holds beyond its original assumptions (the groupoid associated to an irrational flow on the torus provides another example, see Example \ref{exa:rotalg} for details).
Thirdly, a key intermediate result in this setting (Proposition \ref{prop:poincare-type-iso-for-alg-action}) showcases a ``variant'' of Theorem \ref{thmA:poincare-duality} above, more akin to the Thom isomorphism in cohomology than to a proper duality.

\medskip
The paper is organized as follows.
In Section \ref{sec:prelim} we recall a few basic notions and briefly summarize our previous work to set the conventions and background for this note.

In Section \ref{sec:kunneth-formula} we prove the Künneth formula (Theorem \ref{thmA:kunneth}) for the homology of the product of étale groupoids.
In Section \ref{sec:poincare-duality} we show the Poincaré duality-type result (Theorem \ref{thmA:poincare-duality}), and in the last two sections we present some concrete computations for notable examples of topological dynamical systems arising from expanding maps on compact manifolds, and and analogues of solenoids in algebraic number fields.

\subsection*{Acknowledgments}

V.P.:~this research was supported by: Foreign Young Talents' grant (National Natural Science Foundation of China), CREST Grant Number
JPMJCR19T2 (Japan Science and Technology Agency of the Ministry of Education, Culture, Sports, Science and Technology), Marie Skłodowska-Curie Individual Fellowship (project number 101063362).

M.Y.:~this research was funded, in part, by The Research Council of Norway [project 300837].
Part of the work for this project was carried out during M.Y.'s stay at the Research Institute for Mathematical Sciences (RIMS), Kyoto University, Japan.
He thanks N.~Ozawa and others at RIMS for their hospitality.

We would like to thank C.~Bruce and T.~Omland for comments on duality and $S$-integers.
\section{Preliminaries}\label{sec:prelim}

We fix conventions in use throughout the paper.
We only briefly recall definitions and generally follow the treatment in \citelist{\cite{valmak:groupoid}\cite{valmak:groupoidtwo}}.

\subsection{Topological groupoids and homology groups}

We mainly work with second countable, locally compact, Hausdorff groupoids.
Given such a groupoid $G$, we denote its base space by $G^{(0)}$, with structure maps $s, r \colon G \to G^{(0)}$, and the $n$-th
nerve space (for $n\geq 1$) given by
\[
G^{(n)} = \{(g_1, \dots, g_n) \in G^n \colon s(g_i) = r(g_{i+1}) \}.
\]
We say that $G$ is \emph{étale} if $s$ and $r$ are local homeomorphisms, and \emph{ample} if it is étale and its base space is totally disconnected.

\medskip

We consider the homology of étale groupoids as defined by Crainic and Moerdijk \cite{cramo:hom}.
A \emph{$G$-sheaf} is a sheaf $F$ on $G^{(0)}$ endowed with a continuous action of $G$, modeled by a map of sheaves $s^* F \to r^* F$ on $G$.
A $G$-sheaf $F$ is said to be (c-)soft when the underlying sheaf on $G^{(0)}$ has that property, that is, for any (compact) closed subset $S \subset G^{(0)}$ and any section $x \in \Gamma(S, F)$, there is an extension $\tilde x \in \Gamma(X, F)$.

Our standing assumption is that the cohomological dimension of any open sets of $G^{(0)}$ is bounded by some fixed integer $N$: to be precise, $H_c^k(U, F) = 0$ for any open subset $U \subset G^{(0)}$ and any $k > N$.

This is guaranteed when $G^{(0)}$ is a subspace of a metrizable space of (Lebesgue) topological dimension $N$, which covers all the concrete examples we consider.
To see this, observe that the topological dimension of any compact subset $A \subset G^{(0)}$ is bounded by $N$ \cite{MR0482697}*{Section 3.1}, then the Čech cohomology $\check{H}^\bullet(A; F)$ will vanish in degree above $N$.
By paracompactness, $\check{H}^\bullet(A; F)$ agrees with the sheaf cohomology $H^\bullet(A; F) = R^\bullet \Gamma_A(F)$ for the right derived functor of the functor $\Gamma_A(F) = \Gamma(A, F)$, then we can combine \cite{MR0345092}*{Remarque II.4.14.1 and Théorème II.4.15.1} to get the claim.

Let $F$ be a $G$-sheaf.
Then there is a resolution of $F$ as above by c-soft $G$-sheaves, and the \emph{homology with coefficient $F$}, denoted $H_\bullet(G, F)$, is defined as the homology of the total complex of the double complex $(C_i^j)_{0 \le i,j}$ with terms
\[
C_i^j = \Gamma_c(G^{(i)}, s^* F^j),
\]
which has homological degree $i - j$.
More generally, when $F_\bullet$ is a homological complex of $G$-sheaves bounded from below, take a resolution of each $F_j$ by c-soft $G$-sheaves $F_j^k$ as above.
Then the \emph{hyperhomology with coefficient $F_\bullet$}, denoted by $\bH_\bullet(G, F_\bullet)$, is the homology of triple complex with terms
\[
C_{i,j}^k = \Gamma_c(G^{(i)}, s^* F_j^k),
\]
which has homological degree $i + j - k$.

When two étale groupoids $G$ and $H$ are Morita equivalent, there are natural correspondences between the $G$-sheaves and $H$-sheaves inducing an isomorphism of groupoid homology.
In particular, if $f \colon H \to G$ is a Morita equivalence homomorphism, we have
\[
\bH_\bullet(G, F_\bullet) \cong \bH_\bullet(H, f^* F_\bullet)
\]
for any complex $F_\bullet$ of $G$-sheaves as above.
Note that the existence of $f$ is equivalent to $G$ and $H$ being equivalent in the sense studied in \cite{murewi:morita} (see \cite{fkps:hk}*{Prop.~3.10}). For this reason we will often say that $G$ and $H$ are ``equivalent'' without any other specification.

\subsection{Derived functor formalism}\label{sec:der-func-formal}

We briefly recall the derived functor formalism of groupoid homology from \cite{cramo:hom}*{Section 4}.
Let $G$ and $G'$ be étale groupoids, and $\phi\colon G \to G'$ be a continuous groupoid homomorphism.
Then, for each $x \in G^{\prime(0)}$, the \emph{comma groupoid} $x / \phi$ is defined as the groupoid whose objects are the pairs $(y, g')$, where $y \in G^{(0)}$ and $g' \in G_{x}^{\prime \phi(y)}$, and an arrow from $(y_1, g'_1)$ to $(y_2, g'_2)$ is given by $g \in G_{y_1}^{y_2}$ such that $\phi(g) g'_1 = g'_2$.
This is an étale groupoid that comes with a homomorphism $\pi_x \colon x / \phi \to G$.

When $F$ is a $G$-sheaf, we consider a simplicial system of $G'$-sheaves, denoted by $B_\bullet(\phi, F)$, which at the level of stalks is given by
\[
B_n(\phi, F)_x = \Gamma_c((x/\phi)^{(n)}, s^* \pi_x^* F).
\]
Assume that the groupoids have good homological properties to define groupoid homology.
Then the above construction leads to the \emph{left derived functor} $\cL \phi_!$ from a category of homological complexes of $G$-sheaves bounded from below, to a similar category of complexes of $G'$-sheaves.

To be more concrete, let $F_\bullet$ be such a complex of $G$-sheaves.
Then $\cL \phi_! F_\bullet$ is represented by the total complex of the triple complex of $G'$-sheaves with terms $B_i(\phi, F_j^k)$ with homological degree $i + j - k$, where $F_j^\bullet$ is a bounded resolution of $F_j$ by c-soft $G$-sheaves.
This is well defined up to quasi-isomorphism of $G'$-sheaves.
The \emph{$n$-th derived functor}, denoted by $L_n \phi_! F_\bullet$, is the $G'$-sheaf given as the $n$-th homology of $\cL \phi_! F_\bullet$.
By construction the fiber of this sheaf is given by \cite{cramo:hom}*{Proposition 4.3}
\begin{equation}\label{eq:n-th-derived-stalk}
(L_n \phi_! F_\bullet)_x = H_n( x / \phi , \pi_x^* F_\bullet).
\end{equation}
When $G'$ is the trivial groupoid and $\phi$ is the unique homomorphism $G \to G'$, this recovers the definition of $H_n(G, F_\bullet)$.

Besides the pullback functor (the inverse image functor) for sheaves, we will also make use of the \emph{direct image} functor, simply defined as $g_*F(U)=F(g^{-1}(U))$ \cite{MR0345092}.
In the setting of equivariant sheaves, $g_*$ can also be defined, and it is still right adjoint to the pullback functor, see \cite{cramo:hom}*{Section 2.3} for details.
It is worth noting that, if $g$ is proper, then $g_*$ coincides with the functor $g_!$, sometimes called direct image functor with compact supports.
%In the usual setting of topological spaces, $g_!$ is left adjoint to $g^*$ whenever $g$ is étale, however in the case of groupoids the corresponding statement requires additional hypotheses \cite{cramo:hom}*{Remark 5.2}.

\subsection{Smale spaces}
\label{sec:smale-sp}

Many examples of groupoids in this note appear as (reduction of) the stable and unstable groupoid of a \emph{Smale space}.
A Smale space is a certain kind of hyperbolic dynamics modeled on a compact metric space $X$ with a self-homeomorphism $\phi$.
See \cite{put:HoSmale} for precise definition and conventions.
In particular, there are two distinguished equivalence relations on $X$, defined as follows.
\begin{itemize}
\item Two points $x$ and $y$ are \emph{stably equivalent} (denoted $x\sim_s y$) if
\[
\lim_{n\to \infty} d(\phi^{n}(x),\phi^{n}(y)) = 0;
\]
\item similarly, $x$ and $y$ are \emph{unstably equivalent} (denoted $x\sim_u y$) if
\[
\lim_{n\to \infty} d(\phi^{-n}(x),\phi^{-n}(y)) = 0.
\] 
\end{itemize}
The equivalence classes of the stable (resp.~unstable) equivalence relation are called the \emph{stable sets} (resp.~\emph{unstable sets}).

The graph of the stable (resp.~unstable) equivalence relation has a structure of locally compact groupoid with a Haar system \cite{put:algSmale}, that we denote by $R^s(X, \phi)$ (resp.~$R^u(X, \phi)$).
Following the construction detailed in \cite{put:spiel}, we obtain an étale groupoid by restricting $R^u(X,\phi)$ to an appropriate subspace contained in a finite union of stable sets.

\section{Künneth formula}
\label{sec:kunneth-formula}

Suppose $G$ and $H$ are étale groupoids such that groupoid homology is definable, and $\cS$ and $\cT$ are equivariant sheaves of abelian groups over $G^{(0)}$ and $H^{(0)}$ respectively.
Furthermore, denote by $p$ and $q$ the canonical projections from $G \times H$ to $G$ and $H$ respectively.
We define the sheaf $\cS \boxtimes \cT$ as $p^*\cS \otimes q^*\cT$.
Note this is a $G\times K$-equivariant sheaf over $G^{(0)} \times K^{(0)}$.

\begin{theorem}\label{thm:kunneth-formula}
Under the above setting, there is a split short exact sequence
\begin{multline*}
0 \to {\bigoplus_{a + b = k}} H_a(G, \cS) \otimes H_b(H, \cT) \to H_k(G \times H, \cS\boxtimes\cT)\\
\quad \to {\bigoplus_{a + b = k-1}} \Tor(H_a(G, \cS), H_b(H, \cT)) \to 0.
\end{multline*}
\end{theorem}

\begin{proof}
Let us take bicomplexes $A=A_{a,i}$ and $B=B_{b,j}$ computing $H_\bullet(G, \cS)$ and $H_\bullet(H, \cT)$, respectively.
These are obtained from c-soft cohomological complex of sheaves $\tilde\cS^\bullet$ and $\tilde\cT^\bullet$, each quasi-isomorphic to $\cS$ and $\cT$ concentrated at degree $0$.
To obtain a homological complex we invert the degree, so that $A_{a, i} = \Gamma_c(G^{(a)}, \tilde\cS^{-i})$ for example.

Up to the identifications
\[
A_{k,i}\otimes B_{k,j} = \Gamma_c(G^{(k)}, \tilde\cS^{-i}) \otimes \Gamma_c(H^{(k)}, \tilde\cT^{-j}) \cong \Gamma_c((G\times H)^{(k)}, \tilde\cS^{-i} \boxtimes \tilde\cT^{-j}),
\]
the total complex of the triple complex $(A_{k,i} \otimes B_{k,j})$ computes $H_k(G \times H, \cS\boxtimes\cT)$.
The claim follows by a standard argument if we can show that this is quasi-isomorphic to the total complex of the quadruple complex $A \otimes B = (A_{a, i} \otimes B_{b,j})_{a,b,i,j}$.

Now, observe that $A\otimes B$ can be regarded as a bisimplicial object in the category of complexes, by totalizing in the $i$- and $j$-directions.
By an Eilenberg--Zilber type theorem \cite{goja:simp}*{Theorem IV.2.4}, for fixed $q$, the total complex of the bisimplicial group $C_q(a,b) = \bigoplus_{q = i + j} A_{a, i} \otimes B_{b,j}$ is chain homotopic to the Moore complex of the simplicial group $C'_q(k) = \bigoplus_{q = i + j} A_{k, i} \otimes B_{k,j}$.

Now, take double complexes $C_{k,q} = \bigoplus_{k = a + b} C_q(a,b)$ and $C'_{k,q} = C'_q(k)$.
Since the degree $k$ is concentrated in $k \ge 0$ while the degree $q$ is in $q \le 0$, the spectral sequences $E$ and $E'$ associated with filtration by $q$-degree on $\Tot C$ and $\Tot C'$ are regular, in the sense that for any $n$ there is $s(n)$ such that we have $E^r_{p, q} = 0$ for $p + q = n$, $p < s(n)$.
Then the spectral sequences converge to the total homologies, while we have the isomorphisms $E^r_{p,q} \cong E'^r_{p,q}$ for $r \ge 1$ by the above remark.
We thus obtain the assertion.
\end{proof}

As usual, the morphisms in the short exact sequence above are natural in any conceivable sense, however the splitting is not.

\begin{remark}
Matui's result \cite{MR3552533}*{Theorem 2.4} is a special case of the above result in the situation where the groupoids are totally
disconnected and the coefficients are locally constant sheaves $\sfZ$.
(Note that his convention of homology $H_n(G)$ differs from $H_n(G, \sfZ)$ unless $G$ is totally disconnected.)
\end{remark}

\begin{remark}\label{rem:idhomsmale}
In \cite{valmak:groupoidtwo}*{Theorem 5.1} we have identified Putnam's homology groups for Smale spaces \citelist{\cite{put:HoSmale}\cite{val:smale}} with the étale groupoid homology groups considered in this paper, where the groupoid is the unstable equivalence relation associated to a non-wandering Smale space with totally disconnected stable sets (this implies the associated groupoid is ample).
In a companion paper  \cite{valmak:threesmale} to the present one, we remove the hypothesis on the stable sets and prove the identification of homology groups for a general non-wandering Smale space.
Combining this result with Theorem \ref{thm:kunneth-formula} above, we obtain a general Künneth formula for the product of two Smale spaces and their homology groups as defined by Putnam, generalizing \cite{valmak:groupoidtwo}*{Theorem 5.2} and \cite{dkw:dyn}*{Theorem 6.5}.
\end{remark}

\section{Poincaré duality}
\label{sec:poincare-duality}

Suppose we have a locally compact principal groupoid $\tilde G$ such that the map $r \colon \tilde G \to \tilde{G}^{(0)}$ is a fiber bundle whose fibers are (homeomorphic to) $\R^n$.
We assume that there is a generalized transversal $T$ in $\tilde G^{(0)}$, so that $G = \tilde G|_T$ is an étale groupoid. In keeping with our previous installment \cite{valmak:groupoidtwo}*{Definition 2.3}, the notion of generalized transversal we use here is the one introduced in \cite{put:spiel}.
We also assume that groupoid homology for $G$ is definable.

\begin{example}
A motivating example is the unstable groupoid of a Smale space whose unstable classes are homeomorphic to $\R^n$.
In this case, we get the fiber bundle structure and (generalized) transversal from the bracket maps, cf.~\cite{MR1794291}.
\end{example}

Under our assumption, the structure map $s \colon \tilde G \to \tilde G^{(0)}$ is a model of the universal principal $\tilde G$-bundle $E \tilde G \to B \tilde G$.
Then the Baum--Connes conjecture suggests a close relation between $H_\bullet(G, \sfZ)$ and the compactly supported cohomology of
the space $\tilde G^{(0)}$.
Let us make this precise in the framework of groupoid homology.

%Now, let us consider a groupoid $\tilde G$ as in the beginning of this section, and take a transversal $T$ and the associated étale groupoid $G$ as before.
Let us consider the orbit-wise orientation sheaf $\sfo$ on $\tilde G^{(0)}$.
Formally, its stalk at $x$ is given by
\[
\sfo_x = (\medwedge^n T_x \tilde G^x \setminus \{0\}) / \R_+ \cong \R^\times / \R_+ \cong \{1, -1\}.
\]
The étale space structure on the total space $\bigcup_{x \in G^{(0)}} \sfo_x$ is given by the local choice of orientations on the fibers $\tilde{G}^x$.
Namely, whenever $U \subset \tilde{G}^{(0)}$ is an open set such that $\tilde{G}^U$ is homeomorphic to $\R^n \times U$ as a space over $U$, this homeomorphism, together with the standard orientation on $\R^n$, defines points $\sigma_x \in \sfo_x$.
We take the subsets of the form $\{ \sigma_x : x \in U \}$ for such $U$ as a base of the topology on the total space.
Note that $\sfo$ admits a global section (equivalently, it is trivializable) if and only if there is a global orientation on the orbits of $\tilde G$.

By construction, this sheaf has a natural action of $\Z/2\Z$.
Moreover, when $F$ is a sheaf of commutative groups on $\tilde{G}^{(0)}$, we obtain another sheaf $F \times_{\Z/2\Z} \sfo$ of commutative groups, whose fibers are $F_x \times_{\Z/2\Z} \sfo_x$, endowed with the group structure $(a, \sigma) + (b, \sigma) = (a + b, \sigma)$.

\begin{theorem}\label{thm:poincare-duality}
Under the above setting, we have an isomorphism
\[
	H_k(G, \sfZ) \cong H_c^{n-k}(\tilde G^{(0)}, \sfZ \times_{\Z/2\Z} \sfo).
\]
\end{theorem}

\begin{proof}
Consider the $G$-space $E = \tilde G^T$.
Then we have a morphism of groupoid
\[
\phi\colon G \ltimes E \to G
\]
induced by the range map $\tilde G^T \to T$.
We want to apply the constructions in \cite{cramo:hom}*{Section 4} to this setting.
We have an equivalence between $G \ltimes E$ and $\tilde{G}^{(0)}$, induced by the source map $s \colon E \to \tilde G^{(0)}$.
Let us consider the $(G \ltimes E)$-sheaf $F = s^*(\sfZ \times_{\Z/2\Z} \sfo)$.
This relates to the compactly supported cohomology in the claim by the Morita invariance of groupoid homology \cite{cramo:hom}*{Corollary 4.6} and the fact that groupoid homology of a space as a trivial groupoid is just the compactly supported sheaf cohomology up degree inversion, which gives
\begin{equation}\label{eq:morita-equiv-homology}
	H_{p-n}(G \ltimes E, F) \cong H_c^{n-p}(\tilde G^{(0)}, \sfZ \times_{\Z/2\Z} \sfo).
\end{equation}

We will compute the left hand side through its left derived functors $L_k \phi_!$ and get $H_k(G, \sfZ)$.
Recall that the stalks of $L_k \phi_! F$ can be computed as groupoid homology of the groupoid $x/ \phi$ using \eqref{eq:n-th-derived-stalk}.
By our assumption on $\tilde G$, the object space of $x / \phi$ can be identified with the disjoint union of $\tilde G^y$ for $y \in G_x$.
Given objects $g \in \tilde G^y$ and $g' \in \tilde G^z$ in $x / \phi$, there is an arrow from $g$ to $g'$ if and only if $g = g'' g'$ for the unique $g'' \in \tilde G^y_z$.
In particular, $x / \phi$ is Morita equivalent to the space $\tilde G^x \cong \R^n$.

If we restrict the pullback sheaf $\pi_x^* F$ to $\tilde G^x$, we get $s^*(\sfZ \times_{\Z/2\Z} \sfo)$.
Choosing a homeomorphism between $\tilde{G}^x$ and $\R^n$, this is isomorphic to $\sfZ$.
We then have\begin{equation}\label{eq:homol-over-x-mod-phi}
H_k(x/\phi, \pi_x^* F) \cong H_c^{-k}(\R^n, \sfZ) \cong
\begin{cases}
\Z & (k = -n)\\
0 & (\text{otherwise}),
\end{cases}
\end{equation}
where we used the standard orientation on $\R^n$ to get the second isomorphism.
Thus, the $G$-sheaf $L_k \phi_! F$ on $G^{(0)} = T$ has the stalks isomorphic to $\Z$ when $k = -n$, and we have $L_k \phi_! F = 0$ otherwise.
The next step is to check that $L_k \phi_! F$ is isomorphic to $\sfZ$ as a $G$-sheaf.

Since we already know the fibers to be isomorphic to $\Z$, it is enough to check that these isomorphisms can be chosen in a consistent way.
The isomorphism in~\eqref{eq:homol-over-x-mod-phi} was defined on a choice of orientation on $\tilde{G}^x$, or equivalently, a choice of element in the two-point set $\sfo_x$, but we claim that the overall isomorphism is independent of the choice.

Suppose we choose the other element of $\sfo_x$. On one hand, the isomorphism $\pi_x^* F \simeq \sfZ$ is replaced by the negative of the original one. On the other, the isomorphism $H_c^n(\R^n, \sfZ) \simeq \sfZ$ is also affected in the same way if we use an orientation on $\R^n$ different from the first one. Overall, these two changes cancel with each other and give what we wanted.

The $G$-invariance is also straightforward from this presentation.
When $g \in \tilde G^x_y$, the induced map $\tilde{G}^y \to \tilde{G}^x$ gives an isomorphism $H_k(y/\phi, \pi_y^* \sfZ) \to H_k(x/\phi, \pi_x^* \sfZ)$ that is identified with the identity map on $\Z$ up to~\eqref{eq:homol-over-x-mod-phi}.

Now, the Leray-type spectral sequence from \cite{cramo:hom}*{Theorem 4.4}
\[
E^2_{p q} = H_p(G, L_q \phi_! F) \Rightarrow H_{p + q}(G \ltimes E, F)
\]
is degenerate at the $E^2$-sheet for degree reasons.
Thus we get that $H_p(G, \sfZ)$ is isomorphic to $H_{p-n}(G \ltimes E, F)$.
Combined with \eqref{eq:morita-equiv-homology}, we obtain the assertion.
\end{proof}

\begin{remark}[cf.~\cite{MR1950475}*{Section III.5}]
We used the Leray-type spectral sequence in the above proof, but more conceptually, $\cL \phi_! F$ is quasi-isomorphic to the degree shift of $\sfZ$.
To be more precise, there is a zig-zag of quasi-isomorphisms of chain complexes of $G$-sheaves between $\cL \phi_! F$ and $\sfZ[n]$, where $\sfZ[n]$ is the chain complex that has a copy of $\sfZ$ at degree $-n$, and $0$ elsewhere.

In general, suppose that $\cA$ is an abelian category, and consider the category $\cD^-(\cA)$ of homological complexes with terms in $\cA$ that are bounded from below.
We are interested in the case where $\cA$ is the category of $G$-sheaves, hence $\cD^-(\cA)$ is where the left derived functors $\cL_\bullet \phi_!$ are defined.
When an object $X_\bullet$ in $\cD^-(\cA)$ has homology concentrated in degree $k$, we can take the truncation $\tau(X)_\bullet$ defined by
\begin{align*}
\tau(X)_m = \begin{cases}
X_m & (m > k),\\
Z_k(X_\bullet) & (m = k),\\
0 & (m < k).
\end{cases}
\end{align*}
Then the inclusion $\tau(X)_\bullet \to X_\bullet$ and the projection $\tau(X)_\bullet \to H_k(X_\bullet)[-k]$ are quasi\-/isomorphisms of complexes.
\end{remark}

\begin{remark}
Let $(X, \phi)$ be a non-wandering Smale space whose unstable sets are contractible.
In \cite{put:HoSmale}*{Question 8.3.2}, Putnam conjectured that the stable homology $H^s_\bullet(X, \phi)$ is isomorphic to $H^{\bullet}(X)$ up to a degree shift.
In view of Remark \ref{rem:idhomsmale}, Theorem \ref{thm:poincare-duality} answers this conjecture in affirmative if the unstable sets are homeomorphic to $\R^d$ and there is a consistent choice of orientation on these spaces.
As pointed out in \cite{MR3722566}, the original conjecture is false without such orientability, and Theorem \ref{thm:poincare-duality} provides a necessary modification for non-orientable cases.
\end{remark}

\subsection{Examples}

\subsubsection{Substitution tilings}

In the case of groupoids for substitution tilings, the isomorphism of Theorem \ref{thm:poincare-duality} appears in \cite{valmak:groupoid}*{Section 5.2}.

\subsubsection{Anosov diffeomorphisms}

Let $\phi$ be an Anosov diffeomorphism $\phi$ of a compact manifold $X$, with dimension of stable sets $n$.
By the stable manifold theorem \cite{MR0271991} (see also \cite{MR1963683}*{Section 5.6}), the stable set of any point $x \in X$ is an immersed copy of $\R^n$.
Then the monodromy groupoid of stable foliation on $X$ satisfies the assumption for $\tilde G$.

When $X$ agrees with the set of the non-wandering points of $\phi$, we have a non-wandering Smale space $(X, \phi)$, and the above monodromy groupoid is just $R^s(X, \phi)$.

\begin{example}\label{exa:rotalg}
As a concrete example, let us consider the Smale spaces associated to hyperbolic toral automorphisms $(X,\phi)=(\R^2/\Z^2,A)$, where $A$ is a $2$-by-$2$ matrix with integer entries and determinant equal to $1$ (see \cite{put:HoSmale}*{Section 7.4}).
This implies the $\R$-linear endomorphism associated to $A$ descends to a map of the $2$-torus.
Note that $A$ is called ``hyperbolic'' when its eigenvalues $\lambda_1,\lambda_2$ satisfy $\lambda_1<1$, $\lambda_2>1$.

The stable and unstable orbits of $A$ coincide with the lines spanned by the eigenvectors associated to $\lambda_1$ and $\lambda_2$.
Denote by $R^s(X,\phi)$ the stable equivalence relation associated to $(X,\phi)$, suitably reduced to a transversal so that $R^s(X,\phi)$ is étale (in this case a transversal is given for example by the $\lambda_2$-eigenline).
Applying Theorem \ref{thm:poincare-duality}, we obtain 
\begin{equation}\label{eq:homirr}
H_{-1}(R^s(X,\phi))\cong \Z,\qquad H_{0}(R^s(X,\phi))\cong \Z^2,\qquad H_{1}(R^s(X,\phi))\cong \Z.
\end{equation}
Even though the HK conjecture is only formulated for ample groupoids, it is worth pointing out that this homology calculation corresponds (after periodicization) to the $K$-groups of the stable C$^*$-algebra associated to $(X,\phi)$.
Indeed, this algebra is the foliation C$^*$-algebra of the Kronecker flow along the $\lambda_1$-eigenline, hence it is Morita equivalent to the rotation algebra with angle the slope of the eigenline.
The $K$-groups of this algebra is well-known to be $\Z^2$ in both even and odd degree.

The Morita equivalence arises from an equivalence of groupoids as explained in detail in \cite{put:spiel}*{Chapter 3}.
Since the homology groups are Morita invariant \cite{cramo:hom}*{Section 4}, the calculation in \eqref{eq:homirr} is valid for the topological groupoid $\Z\ltimes S^1$ underlying the irrational rotation algebra with angle the slope of the $\lambda_1$-eigenline.
\end{example}

\begin{example}\label{ex:infra-nilmanifold}
Let $M = L / \Gamma$ be an infra-nilmanifold, i.e., a quotient of a nilpotent, connected, and simply connected Lie group $L$ by a torsion-free group $\Gamma$ of affine automorphisms of $L$ such that $\Lambda = L \cap \Gamma$ is a finite index subgroup of $\Gamma$.
Moreover, let $\psi$ be a hyperbolic affine automorphism of $M$ \cite{MR2808264}.
Then $R^u(M, \psi)$ and $R^s(M,\psi)$ satisfy the assumption of Theorem \ref{thm:poincare-duality}.
Indeed, an unstable set of $(M,\psi)$ can be identified with the subspace of the Lie algebra $\mathfrak{l}$ of $L$ spanned by eigenvectors of the ``linear part'' of $\psi$ for corresponding to its eigenvalues bigger than $1$, while a stable set can be identified with the span of the other eigenvectors.
\end{example}

\begin{remark}
If we denote by $S$ the $C^*$-algebra of the stable equivalence relation, Takai \cite{takai:ano} has conjectured that $K_i(S)$ is isomorphic to $K^{i+n}(X)$, which can be understood as an instance of the Baum--Connes conjecture for foliations (for general formulations of the  Baum--Connes conjecture for groupoids see for example \citelist{\cite{val:kthpgrp}\cite{tu:moy}}).
Theorem \ref{thm:poincare-duality} gives an affirmative answer to the homological version of the Takai conjecture.
\end{remark}

\subsubsection{Self-similar action}
\label{sec:self-similar-action}

Another class of examples comes from the theory of self-similar actions, which is also closely related to Example \ref{ex:infra-nilmanifold}.

Let $\Gamma$ be a finitely generated group, and $\phi$ be an injective and surjective contracting virtual endomorphism of $\Gamma$.
Then $\Gamma$ is virtually nilpotent, and admits a self-similar action $(\Gamma, X)$ where the alphabet set $X$ is a system of representatives of $\Gamma / \dom \phi$.
Its \emph{limit $\Gamma$-space} $\cX_{\Gamma, X}$ can be identified with a nilpotent connected and simply connected Lie group $L$ on which $\Gamma$ acts by affine transformations in a proper and cocompact way \cite{MR2162164}*{Section 6.1}.

In general, when $(\Gamma, X)$ is a contracting, recurrent, and regular self-similar action, we have the associated Smale space $\cS_{\Gamma, X}$  (the \emph{limit solenoid} of $(\Gamma, X)$), and its unstable sets can be identified with $\cX_{\Gamma, X}$, see \cite{MR2526786}.
Thus, under the above assumption on $(\Gamma, \phi)$, the unstable groupoid $G = R^u(\cS_{\Gamma, X})$ satisfies the assumption of Theorem \ref{thm:poincare-duality}.

In the next section we look at examples from compact Riemannian manifolds that fall in this setting.

\section{Expanding maps on compact manifolds}

Let us combine the Poincaré duality and transfer maps in homology to obtain a more elaborate computation of groupoid homology.

Suppose that $M$ is an $n$-dimensional connected compact Riemannian manifold and $g \colon M \to M$ is an expanding map.
Then $g$ admits a fixed point $x$, $\Gamma = \pi_1(M, x)$ is a torsion-free group of polynomial growth (hence virtually nilpotent), and $\R^n$ is a universal cover for $M$, see \cite{MR2162164}*{Section 6.1}.

With the virtual endomorphism $\phi$ represented by $g^{-1}$, we are in the setting of Section \ref{sec:self-similar-action}.
Then the Smale space $\cS_{\Gamma, X}$ is given by $Y = \varprojlim_g M$ and the associated self homeomorphism $\phi$ of $Y$.
Again the groupoid $\tilde G = R^u(Y, \phi)$ is equivalent to the étale groupoid $G = \Gamma \ltimes \Omega$ where $\Omega = \varprojlim \Gamma / g^i(\Gamma)$, and we have
\[
H_k(G, \sfZ) \cong H^{n-k}(Y,\sfZ \times_{\Z/2\Z} \sfo).
\]
Let us write this as a group cohomology with coefficient.

Since $G$ is a transformation groupoid, we also have
\[
H_k(G, \sfZ) \cong H_k(\Gamma, C(\Omega, \Z))
\]
with respect to the induced action of $\Gamma$ on $C(\Omega, \Z)$.
As $M$ is a model of the Eilenberg--MacLane space $K(\Gamma, 1)$, $\Gamma$ is a Poincaré duality group \cite{MR1324339}*{Section VIII.10}, and we have
\[
H_k(\Gamma, C(\Omega, \Z)) \cong H^{n-k}(\Gamma, C(\Omega, \Z) \otimes D)
\]
where $D$ is an infinite cyclic group with the ``sign'' representation of $\Gamma$.

Again the equivalence between $\Gamma \ltimes (\Gamma / g^i (\Gamma))$ and $g^i(\Gamma)$ leads to a presentation of these (co)homology groups as inductive limits of the groups $H_k(g^i(\Gamma), \Z) \cong H_k(\Gamma, \Z)$ with connecting maps given by the transfer maps.
This corresponds to the isomorphism
\[
K_\bullet(\Gamma \ltimes C(\Omega)) \cong \varinjlim_i K_\bullet(C^* g^i(\Gamma))
\]
that follows from the Baum--Connes conjecture for coefficients for $\Gamma$, see also \citelist{\cite{deemany:flat}\cite{deeley:hk}} for the case of flat manifolds.

\subsection{Klein bottle}

Let us describe a concrete example in the class of Section \ref{sec:self-similar-action} that arises from a non-orientable surface.
Consider the group
\[
\Gamma =\langle a,b \mid b^{-1}ab=a^{-1}\rangle,
\]
and its action on $\R^2$ given by $a(x,y)=(x+1,y)$, $b(x,y)=(-x,y+1)$.
Let $K$ be the orbit space of this action, which is a model of the Klein bottle space.
Since $\R^2$ is the universal cover of $K$, we have $\Gamma \cong \pi_1(K, [0])$, and $\R^2$ is a model of $E \Gamma$.

Let $g$ be the uniform scaling on $\R^2$ given by $g(x, y) = (3 x, 3 y)$.
This induces to an \emph{expanding endomorphism} of $K$ that fixes $[0]$, which we denote again by $g$.
The associated endomorphism of $\Gamma$ can be written as $g_*(a) = a^3$, $g_*(b) = b^3$ up to the above identification, and the virtual endomorphism $\phi$ represented by $g^{-1}\colon g(\Gamma) \to \Gamma$ satisfies the assumptions of Section \ref{sec:self-similar-action}.

Then, $K$ can be identified with the \emph{limit space} $\cJ_{\Gamma, X}$ of the associated self similar action.
Thus, the associated limit solenoid $\cS_{\Gamma, X}$ is $Y = \varprojlim_g K$, with the induced self-homeomorphism again denoted by $\phi$.
Moreover, the action of $\Gamma$ on the nilpotent Lie group $L$, which appears as the limit $\Gamma$-space as above, is conjugate to the above action of $\Gamma$ on $\R^2$ by the universality of $\R^2$ as the total space of the classifying space for principal $\Gamma$-bundles.

Let us write the coset spaces for the image of powers of $g$ as $\Omega_i= \Gamma / g_*^i(\Gamma)$.
We have natural projection maps $\Omega_{i+1}\to \Omega_i$, and their projective limit $\Omega = \varprojlim_i \Omega_i$.
Then $Y$ can be identified with the homotopy quotient $\Omega \times_K \R^2$.
Thus, the groupoid $\tilde G = R^u(Y, \phi)$ can be identified with the transformation groupoid of $\R^2$ acting on $Y$ by translation.
Taking the image $T$ of $\Omega \times \{(0, 0)\}$ as the transversal, the associated étale groupoid  $G = \tilde G|_T$ is the transformation groupoid of the canonical action of $\Gamma$ on $\Omega$.

Turning to the groupoid homology, at degree $k = 0$ we have
\[
H_0(G, \sfZ) \cong \Z[1/9] = \Z[1/3]
\]
by \cite{valmak:groupoidtwo}*{Theorem 4.1 and Proposition 6.3}.
More generally, as remarked in \cite{valmak:groupoidtwo}*{Section 6.2}, $H_k(G, \sfZ)$ is the direct limit of group homology groups $H_k(\Gamma, \Z)$ where the connecting map is induced by the equivalence between $\Gamma \ltimes \Omega_i$ and $g^i(\Gamma) \cong \Gamma$.
Concretely, these maps are the \emph{transfer maps} of group homology,
\[
H_k(\Gamma, \Z) \to H_k(g^i(\Gamma), \Z),
\]
see \cite{MR1324339}*{Section III.9}.

At degree $k = 1$, let us write
\begin{align*}
H_1(\Gamma, \Z) &= \Gamma^{\ab} \cong \Z \oplus \Z/2\Z,& 
H_1(g(\Gamma), \Z) = g(\Gamma)^{\ab} \cong 3 \Z \oplus \Z/2\Z.
\end{align*}
where the images of $a$ and $g(a)$ become the generator of the summand $\Z/2\Z$, and those of $b$ and $g(b)$ become generators of $\Z$ and $3 \Z$ respectively.
The transfer map is given by
\begin{align*}
H_1(\Gamma, \Z) &\to H_1(g(\Gamma), \Z),&
[a] & \mapsto [g(a)],&
[b] & \mapsto 3 [g(b)],
\end{align*}
see for example \cite{MR1324339}*{Exercise III.9.2}.

Thus, the inductive system $(H_1(g^i(\Gamma), \Z))_i$ can be identified with the constant system $\Z \oplus \Z/2\Z$ whose connecting map is given by $(x, y) \mapsto (3 x, y)$.
In particular we obtain
\[
H_1(G, \sfZ) \cong \Z[1/3] \oplus \Z/2\Z.
\]

For $k \ge 2$, we have $H_k(G, \sfZ) = 0$.
One way to see this is to use the isomorphism
\[
H_k(G, \sfZ) \cong H^{2-k}(Y, \sfZ \times_{\Z/2\Z} \sfo)
\]
given by Theorem \ref{thm:poincare-duality}.
This forces $H_k(G, \sfZ) = 0$ for $k > 2$ by degree reasons, and at $k = 2$ we have
\[
H^0(Y, \sfZ \times_{\Z/2\Z} \sfo) = \Gamma(Y, \sfZ \times_{\Z/2\Z} \sfo) = 0
\]
because there is no global orientation on the orbits of $\tilde G$.
(Alternatively, one can also use the duality between group homology and cohomology for $\Gamma$ to directly check $H_k(\Gamma, \Z) = 0$, see \cite{MR1324339}*{Section VIII.10}.)

\section{Number theoretic generalization of solenoid}\label{sec:ntsol}

We have seen  an example in the previous section where our duality theorem has a fairly straightforward application.
Here we consider a slightly more complicated system, obtained by generalizing the notion of solenoid.
Although Theorem \ref{thm:poincare-duality} cannot be applied directly, the argument in our proof of Poincaré duality will be useful in this setting too, leading us to a statement which is somewhat analogous to the Thom isomorphism.

Let us consider the following class of Smale spaces from \cite{MR1345152}*{Section 7}.
Let $c \in \bar \Q \subset \bC$ be an algebraic number such that $\absv{g c} \neq 1$ for any element $g$ in the absolute Galois group $\Gal(\bar \Q / \Q)$.
Consider the algebraic number field $K = \Q(c)$ generated by $c$.
For each place $v$ of $K$, let us fix an absolute value function $\absv{a}_v$ representing $v$.
We denote the set of finite and infinite places of $K$ by $P^K_f$ and $P^K_\infty$ respectively.

Now, set
\[
P(c) = P^K_\infty \cup \{v \in P^K_f \mid \absv{c}_v \neq 1 \}.
\]
(By our assumption each $v \in P^K_\infty$ satisfies $\absv{c}_v \neq 1$.)
Then the ring
\[
R_c = \{ a \in K \mid \forall v \in P^K_f \smallsetminus P(c) \colon \absv{a}_v \le 1 \}
\]
is a cocompact subring of the direct product of local fields $\prod_{v \in P(c)} K_v$, and the quotient
\begin{equation}\label{eq:ycdual}
Y^{(c)} = \biggl(\prod_{v \in P(c)} K_v\biggr) / R_c
\end{equation}
with respect to the translation action can be identified with the Pontryagin dual of the additive group of $R_c$.

\begin{remark}
To the best of our knowledge, there are some gaps in the proof of this duality presented in \cite{MR1345152}.
An alternative argument is sketched in \cite{ejlo:balian}*{Proposition~6.8}, and we also provide a shorter proof in Appendix \ref{sec:app}.
\end{remark}

The periodic points of the self-homeomorphism
\[
\phi \colon Y^{(c)} \to Y^{(c)}, \quad [a_v]_{v \in P(c)} \mapsto [c a_v]_v
\]
form a dense subset of $Y^{(c)}$ \cite{MR1345152}*{Section 5}, and the $K_v$-direction is contracting (resp.~expanding) for $\phi$ if and only if $\absv{c}_v < 1$ (resp.~$\absv{c}_v > 1$).
Thus, the system $(Y^{(c)}, \phi)$ is a non-wandering Smale space.

Consider the subsets
\[
P^s(c) = \{ v \in P(c) \mid \absv{c}_v < 1 \},\quad
P^u(c) = \{ v \in P(c) \mid \absv{c}_v > 1 \}
\]
of $P(c)$.
Then the stable equivalence class of each point of $Y^{(c)}$ is identified with $\prod_{v \in P^s(c)} K_v$.
Thus, the unstable groupoid of $(Y^{(c)}, \phi)$ is equivalent to the transformation groupoid $R_c \ltimes \prod_{v \in P^s(c)} K_v$.
Similarly, the stable groupoid is equivalent to $R_c \ltimes \prod_{v \in P^u(c)} K_v$.

\subsection{Groupoid homology}

Let us compute the homology of the above groupoids.
In the following, when omitted, the coefficient sheaf for homology is $\sfZ$.

\begin{proposition}\label{prop:poincare-type-iso-for-alg-action}
Let $P^s_f(c) = P^s(c) \cap P^K_f$, and $d = \sum_{v \in P^s(c) \cap P^K_\infty} \dim_\R K_v$.
We then have
\[
H_p\biggl( R_c \ltimes \prod_{v \in P^s(c)} K_v \biggr) \cong H_{p + d}\biggl( R_c \ltimes \prod_{v \in P^s_f(c)} K_v \biggr).
\]
\end{proposition}

\begin{proof}
As in the proof of Theorem \ref{thm:poincare-duality}, consider the groupoid homomorphism
\[
\pi\colon R_c \ltimes \prod_{v \in P^s(c)} K_v \to R_c \ltimes \prod_{v \in P^s_f(c)} K_v 
\]
induced by the projection of base space to the factors labeled by $P^s_f(c)$.
The fibers of this homomorphism can be identified with $\R^d$, and the action of $R_c$ preserves orientation.
We thus have $L_{-d} \pi_! \sfZ \cong \sfZ$ and $L_q \pi_! \sfZ = 0$ for $q \neq -d$ on $ \prod_{v \in P^s_f(c)} K_v$.
Then the Leray-type spectral sequence (see \cite{cramo:hom}*{Theorem 4.4})
\[
E^2_{p q} = H_p\biggl( R_c \ltimes \prod_{v \in P^s_f(c)} K_v, L_q \pi_! \sfZ \biggr) \Rightarrow H_{p+q}\biggl( R_c \ltimes \prod_{v \in P^s(c)} K_v, \sfZ \biggr)
\]
is degenerate at the $E^2$-sheet for degree reasons, and we obtain the claim.
\end{proof}

%Again we have a similar isomorphism of groupoid homology, up to a degree shift, between $R_c \ltimes \prod_{v \in P^s(c)} K_v$ and $R_c \ltimes \prod_{v \in P^s_f(c)} K_v$ for $P^s_f(c) = P^s(c) \cap P^K_f$.

\smallskip
For each $v \in P^s_f(c)$, let $O_v \subset K_v$ denote the corresponding local ring, and $\pi_v \in O_v$ be a generator of its maximal ideal.
Moreover take $m_v \in \N_{>0}$ such that $\absv{c}_v = \absv{\pi_v}_v^{m_v}$, and set
\[
X^{(n)} = \prod_{v \in P^s_f(c)} K_v / \pi_v^{n m_v} O_v.
\]
We have a projective system of the discrete $R_c$-spaces $X^{(n)}$ with proper connecting maps, and
\begin{equation}\label{eq:prlimsol}
R_c \ltimes \prod_{v \in P^s_f(c)} K_v = \varprojlim_n R_c \ltimes X^{(n)}.
\end{equation}
Now, $X^{(n)}$ is a transitive $R_c$-set with stabilizer
\[
\Gamma_n = \{a \in R_c \mid \forall v \in P^s_f(c) \colon \absv{a}_v \le \absv{c}_v^{n} \}.
\]
We thus have $H_\bullet(R_c \ltimes X^{(n)}) \cong H_\bullet(\Gamma_n)$, and
\[
H_k\biggl( R_c \ltimes \prod_{v \in P^s_f(c)} K_v \biggr) \cong \varinjlim_n H_k(R_c \ltimes X^{(n)})
\]
can be identified with the inductive limit of the homology groups $H_k(\Gamma_n)$ with respect to the transfer maps $\theta\colon H_k(\Gamma_n) \to H_k(\Gamma_{n + 1})$ associated with the finite index inclusion $\Gamma_{n + 1} < \Gamma_n$ (see \cite{MR1324339}*{Chapter III, Section 9}).

The transfer map can be made more explicit as follows.
Let us denote by $m_{c^{-1}}$ the isomorphism $\Gamma_{n + 1} \to \Gamma_n$ given by multiplication by $c^{-1}$.

\begin{theorem}\label{thm:solstat}
Let $N = [\Gamma_0 \colon \Gamma_1]$, and $\theta'$ be the endomorphism of $\medwedge^k \Gamma_0$ uniquely extending $N \medwedge^k m_{c^{-1}} \colon \medwedge^k \Gamma_1 \to \medwedge^k \Gamma_0$.
The groupoid homology $H_k( R_c \ltimes \prod_{v \in P^s_f(c)} K_v )$ is isomorphic to the inductive limit $\varinjlim \medwedge^k \Gamma_0$ for the connecting maps $\theta'$.
\end{theorem}

\begin{proof}
First, since $\Gamma_n$ is a torsion-free commutative group, its integral homology $H_\bullet(\Gamma_n, \Z)$ is naturally isomorphic to the exterior algebra $\medwedge^\bullet \Gamma_n$ generated by $\Gamma_n$, see \cite{MR1324339}*{Theorem V.6.4}.
Second, the restriction of $\theta$ to $H_k(\Gamma_{n+1})$ is the multiplication map by $N$ (this is easy to see by comparison with singular homology, see for example \cite{hatcher:alg}*{Section 3G}).
Moreover, $\medwedge^k \Gamma_n$ is torsion-free and $\medwedge^k \Gamma_{n + 1}$ is its finite index subgroup.
Thus the transfer $\medwedge^k \Gamma_n \to \medwedge^k \Gamma_{n + 1}$ is the unique extension of $N \id$ on $\medwedge^k \Gamma_{n + 1}$.
Then the composition of transfer $\medwedge^k \Gamma_0 \to \medwedge^k \Gamma_1$ and the induced isomorphism $\medwedge^k m_{c^{-1}}$ gives $\theta'$, hence we obtain the claim.
\end{proof}

Next let us consider the automorphism of groupoid homology induced by the homeomorphism $\phi$.
As the basepoint we can take $(0)_v \in Y^{(c)}$, which is fixed by $\phi$.
This choice of point is convenient because the action of $\phi$ preserves its stable equivalence class, which can be identified with the multiplication by $c$ on $\prod_{v \in P^s(c)} K_v$, and the groupoid automorphism of $R_c \ltimes \prod_{v \in P^s(c)} K_v$ again given by multiplication by $c$ on all factors.

Note that the isomorphism of Proposition \ref{prop:poincare-type-iso-for-alg-action} is compatible with the automorphism induced by $\phi$, and the analogous one on $R_c \ltimes \prod_{v \in P^s_f(c)} K_v$, up to multiplying by $-\id$ certain factors of $K_v$ when $K_v \cong \R$.
This is because multiplication by $c$ on $K_v$, for Archimedean $v$, is properly homotopic to the identity map $\id$, or possibly to $-\id$ when $K_v \cong \R$.

\begin{remark}
In the above argument we used the invariance of groupoid homology under proper homotopy, which can be obtained by considering the following chain of isomorphisms: 
\[ H_k( R_c \ltimes \prod K_v,\sfZ )\cong H_k(R_c,C( \prod K_v ,\Z))\cong H_k(B R_c,C( \prod K_v ,\Z)),
\]
where we have group homology in the middle, and the last group is the homology of the classifying space with coefficient the local system induced by $\prod K_v$.
See also \cite{kascha:shman}*{Proposition 2.7.5} for a proof of homotopy invariance in the setting of sheaves.
\end{remark}

\bigskip
In  \cite{put:HoSmale}*{Theorem 6.1.1}, Putnam gave a remarkable analogue of Lefschetz formula for non-wandering Smale spaces $(X, \phi)$, which gives the number of periodic points as
\[
\absv{\{ x \in X \mid \phi^n(x) = x \}} = \sum_k (-1)^k \Tr_{H^s_k(X, \phi) \otimes \Q}((\phi^{-n})_*)
\]
using the transformation on his stable homology $H^s_k(X, \phi)$ induced by $\phi^{-n}$.

For the Smale spaces $(Y^{(c)}, \phi)$, we can make both sides more explicit.

\begin{proposition}
For $n \ge 1$, the fixed points for $\phi^n$ on $Y^{(c)}$ are bijectively parameterized by 
\begin{equation*}%\label{eq:perpo}
(c^n- 1)^{-1} \cO_K / (\cO_K \cap (c^n - 1)^{-1} \cO_K ).
\end{equation*}
\end{proposition}

\begin{proof}
Since $Y^{(c)}$ is defined as the quotient space $\prod K_v/R_c$, the points with period $n$ have coordinates $x$ (for all places $v$) such that $(c^n - 1) x$ belongs to $R_c$.
These points are parameterized by $(c^n - 1)^{-1} R_c) / (R_c \cap (c^n - 1)^{-1} R_c)$.
We can further simplify this quotient and arrive at the claim by writing $R_c$ as an increasing union of copies of $\cO_K$, and noticing the compatibility of the quotient with the inductive structure.
\end{proof}

\begin{proposition}
The automorphism of $H_k(R_c \ltimes \prod_{v \in P^s_f(c)} K_v) \otimes \Q$ induced by $\phi^{-1}$ has trace equal to that of $N \medwedge^k m_{c^{-1}}$ acting on $\medwedge^k \Gamma_0 \otimes \Q$.
\end{proposition}

\begin{proof}
Given $a \in \medwedge^k \Gamma_0$, let us write $a_j$ for the element of $\varinjlim \medwedge^k \Gamma_0$ represented by $a$ in the $j$-th copy of $\medwedge^k \Gamma_0$.
Unpacking the correspondence in Theorem \ref{thm:solstat}, the automorphism on $\varinjlim \medwedge^k \Gamma_0$ induced by $\phi$ is the map $f(a_j) = a_{j + 1}$ for $a \in \medwedge^k \Gamma_0$.
Thus $\phi^{-1}$ induces the transform
\[
\tilde f(a_j) = a_{j - 1} = N (\medwedge^k m_{c^{-1}})(a)_j,
\]
hence we obtain the claim.
\end{proof}

\subsection{Comparison with \texorpdfstring{$K$}{K}-groups}

The discussion above has a straightforward counterpart in $K$-theory for the corresponding groupoid $C^*$-algebras.
As we are working with the transformation groupoid $R_c \ltimes \prod_{v \in P^s(c)} K_v$, the algebra of interest is the crossed product $C^*$-algebra $C_0(\prod_{v \in P^s(c)} K_v)\rtimes  R_c$.
(Of course, the stable case is analogous.)

Let us first see the analogue of Proposition \ref{prop:poincare-type-iso-for-alg-action} in $K$-theory.

Let $\bC_d$ denote the (complex) Clifford algebra of the Euclidean space $\R^d$, which is a $\Z_2$-graded C$^*$-algebra.
Then there is a Dirac morphism, given by a (strictly) equivariant unbounded Fredholm module $D$ between $C_0(\R^d) \otimes \bC_d$ and $\bC$ for the translation action of $\R^d$ on $C_0(\R^d)$, hence a class
\[
[D] \in \KK^{\R^d}(C_0(\R^d) \otimes \bC_d, \bC) = \KK^{\R^d}_d(C_0(\R^d), \bC).
\]
This class $[D]$ is invertible by Connes's Thom isomorphism theorem in $\KK$-theory \cite{skandalis:fack}.
This can be also interpreted as the strong Baum--Connes conjecture for $\R^d$ \cite{higkas:bc} (see also \cite{val:shi}*{Corollary 2.3}).

Now, with $d$ as in Proposition \ref{prop:poincare-type-iso-for-alg-action}, take the group embedding of $R_c$ into $\R^d$ corresponding to the completion at infinite places.
Since $D$ is given by a strictly $\R^d$-equivariant Fredholm module, we obtain a class
\[
[D_c] \in \KK^{R_c}_d(C_0(\R^d),\mathbb{C}).
\]

\begin{proposition}\label{prop:K-theoretic-deg-shift-isom}
The class $[D_c]$ induces an isomorphism 
\[
K_{\bullet} \biggl( C_0\biggl(\prod_{v \in P^s(c)} K_v\biggr)\rtimes R_c\biggr)\cong  K_{\bullet+d} \biggl( C_0\biggl(\prod_{v \in P^s_f(c)}K_v\biggr) \rtimes R_c \biggr).
\]
\end{proposition}

\begin{proof}
Let us set $A= C_0(\prod_{v \in P^s(c)} K_v)$ and $B=C_0(\prod_{v \in P^s_f(c)} K_v$), so that we have $A \cong B \otimes C_0(\R^d)$.
Then $f=\mathrm{id}_B\otimes [D_c]$ defines an invertible class in $\KK^{R_c}(A\otimes \bC_d,B)$.

Taking Kasparov's descent (see \cite{kas:descent}*{page 172}), we see that the crossed product algebras $(A \otimes \bC_d) \rtimes R_c = A \rtimes R_c \otimes \bC_d$ and $B \rtimes R_c$ are KK-equivalent.
Taking into account the degree-shift induced by the Clifford algebra, we obtain the claim.
\end{proof}

Next let us present an analogue of Theorem \ref{thm:solstat}.
From the inverse limit presentation of \eqref{eq:prlimsol}, we obtain an inductive limit structure
\begin{equation}\label{eq:prlimsolc}
C_0\biggl(\prod_{v \in P^s_f(c)} K_v\biggr) \rtimes  R_c\cong \varinjlim_n C_0(X^{(n)})\rtimes R_c.
\end{equation}

For a $\Z$-module $A$, let us define $\medwedge^{[i]}A$ as the direct sum of all exterior powers of $A$ whose degree is congruent to $i$ modulo $2$.
Analogously, we define $H_{[i]}(G,\sfZ)$ as the direct sum of homology groups with corresponding degree parity.

\begin{theorem}\label{thm:hksolthm}
There is an isomorphism 
\[
K_i\biggl(C_0\biggl(\prod_{v \in P^s_f(c)} K_v\biggr) \rtimes  R_c\biggr)\cong \varinjlim \medwedge^{[i]} \Gamma_0,
\]
with connecting maps given by the unique extension of 
\[
N\bigwedge^{[i]}m_{c^{-1}}\colon \bigwedge^{[i]} \Gamma_1\to \bigwedge^{[i]}\Gamma_0
\]
as in Theorem \ref{thm:solstat}.
\end{theorem}

\begin{proof}
On one hand, taking $K$-groups is compatible with taking inductive limits of $C^*$-algebras.
On the other, $X^{(n)}$ is a transitive $R_c$-set with stabilizer equal to $\Gamma_n$.
Combining these and \eqref{eq:prlimsolc}, we have
\begin{equation*}
K_i\biggl(C_0\biggl(\prod K_v\biggr) \rtimes  R_c\biggr)\cong \varinjlim_n K_i(C^*(\Gamma_n)).
\end{equation*}

Next, let us identify $K_i(C^*(\Gamma_n))$ with $\medwedge^{[i]} \Gamma_n$.
Consider the subgroup
\[
\Gamma_n^k = \{a\in\Gamma_n \mid \forall v\in  P^u_f(c) \colon \absv{a}_v \le \absv{c}_v^{k} \} < \Gamma_n.
\]
We then have $\bigcup_k \Gamma^k_n = \Gamma_n$.
Each group $\Gamma^k_n$ is finitely generated, because it is generated by a finite union of groups which are isomorphic to the ring of integers $\cO_K$, which is free abelian of rank equal to the degree of $K$.
Thus $\Gamma^k_n$, being torsion free and finitely generated commutative group, is also free abelian.
From this we obtain
\[
K_i(C^*(\Gamma^k_n)) \cong \medwedge^{[i]}\Gamma^k_n.
\]
As both sides are compatible with colimit, we obtain
\[
K_i(C^*(\Gamma_n)) \cong \medwedge^{[i]} \Gamma_n.
\]

It remains to identify the connecting map
\begin{equation}\label{eq:K-group-connecting-map}
K_i(C^*(\Gamma_n)) \to K_i(C^*(\Gamma_{n + 1})).
\end{equation}

Generally, suppose one has a finite index inclusion of discrete groups $H < G$.
Consider the $G$-equivariant map $\bC \to C(G/H)$, and the induced map $C_r^*(G)\to C(G/H)\rtimes_r G$.
Then the map of the reduced group C$^*$-algebras
\[
K_\bullet(C^*_r(G)) \to K_\bullet(C(G/H)\rtimes_r G) \cong K_\bullet(C^*_r(H))
\]
can be computed as follows.
We fix a system of representatives $(g_i)_{i = 1}^N$ of $G/H$.
Then we get an isomorphism
\[
C(G/H) \rtimes_r G \cong M_N(\bC)\otimes C_r^*(H)
\]
that sends $\delta_{g_i} \in C(G/H)$ to $e_{i,i} \in M_N$ and $\lambda_g \in C^*_r(G)$ to $\sum_i e_{\sigma_g(i),i}\otimes \delta_{h(i,g)}$, where $\sigma_g(i)$ and $h(i,g)$ are characterized by the relation $g g_i = g_{\sigma_g(i)}h(i,g)$.

If we use $G=\Gamma_n$ and $H=\Gamma_{n+1}$, by commutativity, the image of $\lambda_g \in C^*_r(\Gamma_{n+1})$ under the map is a diagonal matrix embedding of $g$.
Thus, \eqref{eq:K-group-connecting-map} is an extension of the multiplication by $N$ on $K_\bullet(C^*(\Gamma_{n+1}))$.
Taking into account the isomorphisms $\medwedge^{[i]} \Gamma_n \cong \medwedge^{[i]} \Gamma_0$ given by composition of $m_{c^{-1}}$, we obtain the claim.
\end{proof}

\begin{corollary}\label{cor:HK-conj-for-num-theretic-solenoid}
The HK conjecture holds for the groupoid $R_c \ltimes \prod_{v \in P^s_f(c)} K_v$.
\end{corollary}

\begin{remark}
While the original HK-conjecture \cite{MR3552533}*{Conjecture 2.6} was formulated for ample groupoids, it has an obvious generalization to the setting of étale groupoids.
Unstable and stable groupoids of Smale spaces, reduced to transversal subspaces, give rich example of such groupoids.
While there are counterexamples to the conjecture (in the ample case) \citelist{\cite{deeley:hk}\cite{scarparo:hk}}, Corollary \ref{cor:HK-conj-for-num-theretic-solenoid} holds for the stable and unstable groupoids of the Smale space $(Y^{(c)}, \phi)$.
In this case Propositions \ref{prop:poincare-type-iso-for-alg-action} and \ref{prop:K-theoretic-deg-shift-isom} can be viewed as a reduction step to the ample case.
\end{remark}

\begin{remark}
We note that there exist maps $\mu_i\colon H_i(G)\to K_i(C^*_rG)$  for $i = 0, 1$ in the case of ample groupoids, and they have been studied in \cites{bdgw:matui,matui:hklon}.
For higher degrees, one can still construct analogous maps by looking at kernels of higher differentials from the associated spectral sequence from~\cite{valmak:groupoid}.
\end{remark}

Before looking at some examples, let us identify the positive cone of the $K_0$-group appearing in Theorem \ref{thm:hksolthm}.
We are going to need the following cancellation theorem.

\begin{lemma}[\cite{MR1249482}*{Theorem 9.1.2}]\label{lem:brank}
Let $X$ be a CW-complex of dimension $n$, and set $s=\lceil n/2\rceil$.
Every complex vector bundle $E$ over $X$ of rank $r \ge s$ splits as $E\cong E^\prime \oplus \varepsilon^{r-s}$, where $E'$ is a vector bundle over $X$ and $\varepsilon^k$ is the trivial bundle of rank $k$.
\end{lemma}

Let us write $\Z^+[1/N]$ for the subset of positive numbers in $\Z[1/N]$ (as a subset of $\R$).

\begin{proposition}\label{prop:pcone}
Suppose that $N > 1$.
Under the identification in Theorem \ref{thm:hksolthm}, we have 
\[
K_0\biggl(C_0\biggl(\prod_{v \in P^s_f(c)} K_v\biggr) \rtimes  R_c\biggr)^+ \cong \Z^+[1/N]\oplus \bigoplus_{i\geq 1} \varinjlim \medwedge^{2i} \Gamma_0.
\]
\end{proposition}

\begin{proof}
Let us write $A_n = C_0(X^{(n)}) \rtimes R_c$, and $\phi^{n+1}_n$ be the connecting map of the inductive system~\eqref{eq:prlimsolc}.
We further write
\[
A_\infty = C_0\biggl(\prod_{v \in P^s_f(c)} K_v\biggr) \rtimes  R_c.
\]

Let $x \in K_0(A_\infty)$ be a positive element.
By the density of $\bigcup_n A_n$ in $A_\infty$, we find a positive element $x_n^\prime\in K_0(A_n)$.
Using the Morita equivalence between $R_c\ltimes X^{(n)}$ and $\Gamma_n$, we can further replace $x_n^\prime$ with $x_n\in K_0(C^*\Gamma_n)$, which is still positive as the Morita equivalence preserves positivity.
Again by density of $\bigcup_k C^* \Gamma_n^k$ in $C^* \Gamma_n$ for the subgroups $\Gamma_n^k$ from the proof of Theorem~\ref{thm:hksolthm}, we get a projection $p \in M_d(C^* \Gamma_n^k)$ whose class maps to $x$.

Now, under the isomorphism $K_0(C^*\Gamma^k_{n})\cong\medwedge^{[0]}\Gamma^k_{n}$, the canonical tracial state $\tau_n$ of the group algebra $C^* \Gamma^k_n$ picks up the constant term in $\medwedge^0 \Gamma^k_n \cong \Z$, which must be a positive integer for $x$.
When we increase $n$, the action of $\phi^{n+1}_n$ on the constant term is multiplication by $N$.
We thus obtain the degree $0$ part of $x$ to be in $\Z^+[1/N]$ as claimed.

For the converse implication, let us take an element $x \in K_0(A_\infty)$ whose degree $0$ term belongs to $\Z^+[1/N]$.
By reasoning as above, we find its representative $x^k_n\in K_0(C^*\Gamma^k_n)$, whose degree $0$ part is positive in $\medwedge^0 \Gamma^k_n \cong \Z$.

Recall that $\Gamma^k_n$ is free abelian.
Moreover, its rank $d$ is independent of $n$ and (big enough) $k$, as we have $d = \dim_\Q \Gamma_n \otimes \Q$.
We thus have $C^*\Gamma^k_n\cong C(\bT^d)$ for all $k$ and $n$.
The degree $0$ part $\tau_n(x^k_n)$ of $x^k_n$ agrees with the rank of the corresponding vector bundle on $\bT^d$ in the latter picture.

Now, take $m > 0$ big enough such that $r = \tau_{n+m}(\phi_{n,\ast}^{n+m}(x_n^k)) = N^{m} \tau_n(x)$ is larger than $s = \lceil d/2\rceil$.
Then the corresponding element in $K_0(C^* \Gamma_{n + m}^{\ell}) \cong K^0(\bT^d)$ is represented by $[E] - [\varepsilon^t]$, where $E$ is a vector bundle of rank $r + t$.
By Lemma \ref{lem:brank}, we have
\[
[E]-[\varepsilon^t] = [E^\prime]+[\varepsilon^{r + t - s}] -[\varepsilon^t] = [E^\prime\oplus \varepsilon^{r - s}],
\]
which represents a positive element.
This completes the proof.
\end{proof}

\subsection{Degree \texorpdfstring{$1$}{1} case}

The case of $K(c) = \Q$, i.e., $c \in \Q$, is considered by Burke and Putnam \cite{buput:ntsol}, where they computed the stable and unstable Putnam homology groups.
To simplify the presentation let us consider the case of two prime factors, as follows.

Let $p < q$ be two prime numbers, and put $M = p q$, $c = \frac{q}{p}$, $R_c = \Z[1/M]$.
Our Smale space is given by the compact space 
\[
X = Y^{(c)} = (\R \times \Q_p \times \Q_q) / R_c
\] 
and the self-homeomorphism $\phi([x, y, z]) = [ c x, c y, c z ]$.
Consequently the étale groupoid $G = R^u(X, \phi)|_{X^s(x_0)}$ is the transformation groupoid $R_c \ltimes \Q_q$, while $G' = R^s(X, \phi)|_{X^u(x_0)}$ is the transformation groupoid $R_c \ltimes (\R \times \Q_p)$.

Now the subgroup $\Gamma_n < R_c$ is given by
\[
\Gamma_n = \biggl\{ \frac{a q^n}{p^k} \mid a \in \Z, k \in \N \biggr\},
\]
hence we have $N = q$.
The exterior algebra $\medwedge^\bullet \Gamma_0$ is
\begin{align*}
\medwedge^0 \Gamma_0 &= \Z,&
\medwedge^1 \Gamma_0 &= \Z\biggl[\frac1p\biggr],&
\medwedge^k \Gamma_0 &= 0 \quad (k > 1).
\end{align*}
The map $N \medwedge^k m_{c^{-1}}$ is multiplication by $q$ for $k = 0$ and multiplication by $p$ for $k = 1$.
We thus obtain
\begin{align}\label{eq:R-c-Q-q-homology}
H_0(R_c \ltimes \Q_q) &\cong \Z\biggl[\frac1q\biggr],&
H_1(R_c \ltimes \Q_q) &\cong \Z\biggl[\frac1p\biggr],&
H_k(R_c \ltimes \Q_q) &= 0 \quad (k \neq 0, 1).
\end{align}
This gives a description of the groupoid homology $H_k(G, \sfZ)$.

As for the stable equivalence relation, Proposition \ref{prop:poincare-type-iso-for-alg-action} gives
\[
H_k(R_c \ltimes (\R \times \Q_p)) \cong H_{k + 1}(R_c \ltimes \Q_p).
\]
This, together with \eqref{eq:R-c-Q-q-homology} (switching the role of $p$ and $q$) gives
\begin{align*}
H_{-1}(G', \sfZ) &\cong \Z\biggl[\frac1p\biggr],&
H_0(G', \sfZ) &\cong \Z\biggl[\frac1q\biggr],&
H_k(G', \sfZ) &= 0 \quad (k \neq 0, -1).
\end{align*}

\begin{remark}
In view of Remark \ref{rem:idhomsmale}, the computation of Putnam's stable homology for $(Y^{(c)}, \phi)$ with $c \in \Q$ carried out in \cite{buput:ntsol} already gives the formulas \eqref{eq:R-c-Q-q-homology}.
The method presented here is arguably more direct and does not require a deep understanding of the dynamics of the system, on the other hand the method in \cite{buput:ntsol} is intimately tied to Markov partitions and allows to understand the system in terms of symbolic coding.
\end{remark}

\subsection{Degree \texorpdfstring{$2$}{2}, imaginary case}

Next consider the case of quadratic extensions.
To illustrate the situation associated with prime ideals that are not singly generated, we look at the case $K = \Q(\sqrt{-5})$, so that $\cO_K = \Z[\sqrt{-5}]$ is not a principal ideal domain.

To achieve this, let us take
\[
c = \frac{1 + \sqrt{-5}}{2}.
\]
The relevant finite places of $K$ correspond to the prime ideals
\begin{align*}
\frp_1 &= (2, 1 + \sqrt{-5}),&
\frp_2 &= (3, 1 + \sqrt{-5}).
\end{align*}
For $i = 1, 2$, we write $v_i$ for the corresponding place, $\absv{a}_i$ for the absolute value, $K_i$ for the completed local field, $O_i < K_i$ for the local ring, and $\pi_i \in O_i$ for the uniformizer.
Thus, $x \in \cO_K$ has absolute value $\absv{x}_i = \absv{\pi_i}_i^m$ if and only if the principal ideal $(x) < \cO_K$ decomposes as
\[
(x) = \frp_i^m \fra
\]
for some ideal $\fra < \cO_K$ such that $\frp_i \mathbin{\not|} \fra$.
Then
\begin{align*}
(2) &= \frp_1^2,&
(1 + \sqrt{-5}) &= \frp_1 \frp_2
\end{align*}
imply that we have
\begin{align*}
\absv{c}_1 &= \absv{\pi_1}_1^{-1},&
\absv{c}_2 &= \absv{\pi_2}_2.
\end{align*}
Writing $v_\infty$ for the unique infinite place, we have
\begin{align*}
P^s(c) &= \{v_2\},&
P^u(c) &= \{v_\infty, v_1\}.
\end{align*}

From this we get
\[
R_c = \biggl\{ \frac{x}{2^k (1 + \sqrt{-5})^k} \biggm\vert x \in \cO_K, k \in \N \biggr\}.
\]
(We cannot have $3$ in the denominator as the prime ideal $\frp_3 = (3, 1 - \sqrt{-5})$ satisfies $(3) = \frp_2 \frp_3$, hence $\frac13$ would have absolute value bigger than $1$ for the corresponding absolute value.)
Thus, the Smale space $(Y^{(c)}, \phi)$ is given by
\[
Y^{(c)} = (\bC \times K_1 \times K_2) / R_c,
\]
and the diagonal action of $c$ gives
\begin{align*}
R^u(X, \phi)|_{X^s(x_0)} &\cong R_c \ltimes K_2,&
R^s(X, \phi)|_{X^u(x_0)} &\cong R_c \ltimes (\bC \times K_1)
\end{align*}
for $x_0 = (0, 0, 0)$.

To compute its groupoid homology for unstable groupoid, the general constructions from above give
\begin{align*}
\Gamma_0 &= \biggl\{ \frac{x}{2^k} \biggm\vert x \in \cO_K, k \in \N \biggr\},&
\Gamma_1 &= \biggl\{ \frac{x}{2^k} \biggm\vert x \in \frp_2, k \in \N \biggr\}.
\end{align*}
Thus, $\Gamma_0 / \Gamma_1 \cong \cO_K / \frp_2$ and we have $N = [\Gamma_0 : \Gamma_1] = 3$.
As for the exterior algebra, we have
\begin{align*}
\medwedge^2 \Gamma_0 &= \biggl\{ \frac{1 \wedge \sqrt{-5}}{2^j} k \biggm\vert j \in \N, k \in \Z \biggr\},&
\medwedge^k \Gamma_0 &= 0 \quad (k > 2).
\end{align*}
Similarly, we have
\[
\medwedge^2 \Gamma_1 = \biggl\{ \frac{3 \wedge (1 + \sqrt{-5})}{2^j} k \biggm\vert j \in \N, k \in \Z \biggr\}
\]
as $\frp_2$ is free of rank $2$ as a $\Z$-module, with basis $3$ and $1 + \sqrt{-5}$.

Next let us identify the extensions of
\[
N \medwedge^k m_{c^{-1}} \colon \medwedge^k \Gamma_1 \to \medwedge^k \Gamma_0
\]
to $\medwedge^k \Gamma_0$, and the inductive limit with respect to this map.
When $k = 0$, by $N = 3$ it is the multiplication by $3$ on $\Z$, hence the limit is $\Z[\frac13]$.
When $k = 1$, it is multiplication by $6 / (1 + \sqrt{-5}) = 1 - \sqrt{-5}$, hence the limit is
\[
\varinjlim \Gamma_0 = \Z\biggl[\sqrt{-5}, \frac12, \frac1{1 - \sqrt{-5}}\biggr].
\]

When $k = 2$, this map is
\[
\medwedge^2 \Gamma_1 \to \medwedge^2 \Gamma_0, \quad \frac{3 \wedge (1 + \sqrt{-5})}{2^j} k \mapsto 3 \frac{(1 - \sqrt{-5}) \wedge 2}{2^j} k = 6 \frac{1 \wedge \sqrt{-5}}{2^j} k.
\]
On the other hand, in $\medwedge^2 \Gamma_0$ we have
\[
\frac{3 \wedge (1 + \sqrt{-5})}{2^j} k = 3 \frac{1 \wedge \sqrt{-5}}{2^j} k.
\]
Hence the above map extends to multiplication by $2$ on $\medwedge^2 \Gamma_0$.
Thus, the limit is given by
\[
\varinjlim \medwedge^2 \Gamma_0 \cong \Z\biggl[ \frac12 \biggr].
\]

Summarizing, the étale groupoid $G = R^u(X, \phi)|_{X^s(x_0)}$ has the integral homology groups
\begin{align*}
H_0(G) &\cong \Z\biggl[\frac13\biggr],&
H_1(G) &\cong \Z\biggl[\sqrt{-5}, \frac12, \frac1{1 - \sqrt{-5}}\biggr],&
H_2(G) &\cong \Z\biggl[\frac12\biggr],& 
H_k(G) &= 0 \quad (k \neq 0, 1, 2),
\end{align*}
on which the induced action of $\phi^{-1}$ is respectively by multiplication by $3$, $1 - \sqrt{-5}$, and $2$.

\medskip
The homology of $G' = R^s(X, \phi)|_{X^u(x_0)}$ can be computed in a similar way as above, combined with Proposition \ref{prop:poincare-type-iso-for-alg-action}, and we get
\begin{align*}
H_{-2}(G^\prime) &\cong \Z\biggl[\frac12\biggr],&
H_{-1}(G^\prime ) &\cong \Z\biggl[\sqrt{-5}, \frac12, \frac1{1 + \sqrt{-5}}\biggr],&
H_0(G^\prime) &\cong \Z\biggl[\frac16\biggr],& 
\end{align*}
on which the induced action of $\phi$ is respectively by multiplication by $2$, $1 + \sqrt{-5}$, and $3$, and $H_k(G^\prime) = 0$ for $k \neq 0, -1, -2$.

\subsection{Degree \texorpdfstring{$2$}{2}, real case}

Let us next consider a case with nontrivial unit.
We look at a unit in real quadratic field.

Concretely, let us take
\[
c = \frac{1 + \sqrt{5}}2,
\]
hence $K = \Q(\sqrt{5})$.
In this case we find that $c$ is invertible in $\cO_K = \Z[c]$, with $-c^{-1}$ being the Galois conjugate of $c$ (the nontrivial automorphism of $K$ is given by $a+b\sqrt{5}\mapsto a-b\sqrt{5}$).

Then the relevant places are the Archimedean places $v_\infty$ and $v_\infty'$, for which the corresponding absolute values are the usual one and its twist by the automorphism of $K$ that sends $c$ to $-c^{-1}$.
Thus, we have $R_c = \cO_K$, which is isomorphic to $\Z^2$ as a commutative group.
Then $Y^{(c)}$, being its Pontryagin dual, can be identified with $\bT^2$.

As a Smale space, the corresponding homeomorphism $\phi$ is induced by the matrix presentation of multiplication by $c$ on $\cO_K \cong \Z^2$, that is,
\[
\begin{bmatrix}
0 & 1 \\
1 & 1
\end{bmatrix}.
\]
This way we obtain the hyperbolic toral automorphism as in Example \ref{exa:rotalg}.

In this case, Theorem \ref{thm:poincare-duality} for $\tilde G = R^u(X, \phi)$ and $G = R^u(X, \phi)|_{X^s(x_0)}$ gives
\[
H_k(G) \cong \begin{cases}
\Z & (k = -1, 1)\\
\cO_K \cong \Z^2 & (k = 0)\\
0 & (\text{otherwise})
\end{cases}
\]
with $\phi^{-1}$ acting by $\pm 1$ for $k = \pm 1$, and by $-c^{-1}$ on $\cO_K$ for $k = 0$.
We also have an analogous presentation of homology for $G = R^s(X, \phi)|_{X^u(x_0)}$.

\begin{remark}
In general, the conjugacy classes of hyperbolic matrices in $\SL_n(\Z)$ with distinct eigenvalues bijectively correspond to the ideal classes in the integer rings of certain totally real fields, see \citelist{\cite{MR1503108}\cite{MR30491}\cite{MR735415}}.
\end{remark}

\appendix
\section{Duality for \texorpdfstring{$S$}{S}-integers}
\label{sec:app}

Let $K$ be a number field.
We denote a place of $K$ by $v$, and the associated absolute value by $\absv{a}_v$ for $a \in K$.
The associated completed local field is denoted by $K_v$, while (for a finite place $v$) its maximal compact subring is denoted by $O_v$.
Then the adele ring of $K$ is given by the restricted product
\[
\bA_K = \prod_{v \in P_\infty^K} K_v \times \prod_{v \in P_f^K} (K_v, O_v)
\]
of the $K_v$ relative to $O_v$.

Let $S$ be a finite set of places of $K$ which contains all the infinite places.
We denote by $R_S$ the ring of $S$-integers, i.e.,
\[
R_S = \{a \in K \mid \forall v \not\in S \colon \absv{a}_v \le 1 \},
\]
and by $\bA_{K,S}$ the ring of $S$-adeles, i.e.,
\[
\bA_{K,S} = \prod_{v \in S} K_v \times \prod_{v \not\in S} O_v.
\]
In the setting of Section \ref{sec:ntsol}, we take $S = P(c)$ so that $R_c = R_S$.

When $G$ is a locally compact commutative group, we denote its Pontryagin dual by $\hat{G}$.
Our goal is to establish the following.

\begin{theorem}\label{thm:dual-of-R-S}
The dual $\hat R_S$ of the additive group of $R_S$ is isomorphic to 
\[
G_S = \Bigl(\prod_{v \in S} K_v\Bigr) / R_S.
\]
\end{theorem}

Let us recall some standard facts (see for example \cite{weil:bnt}*{Section II.5}).

\begin{proposition}\label{prop:sdk}
For any place $v$, there exists a self-duality pairing on $K_v$ such that the complement of $O_v$ is identified with $O_v$ itself.
\end{proposition}

\begin{corollary}\label{cor:dual-of-O-v}
The Pontryagin dual of $O_v$ is isomorphic to $K_v / O_v$.
\end{corollary}

It is worth emphasizing that the isomorphism above is not canonical: it dependes on a specific choice of pairing for $K_v$.
We will consider this choice fixed for the rest of the appendix. Having Proposition \ref{prop:sdk}, it is easy to prove the standard self-duality
$\bA_K\cong \hat{\bA}_K$ of the adele ring. Then, considering the extension $K\to\bA_K\to\bA_K/K$ yields the following well-known result: 

\begin{proposition}
The dual group $\hat K$ is isomorphic to $\bA_K / K$.
\end{proposition}

\begin{proposition}[\emph{strong approximation}, cf.~\cite{MR0861410}*{Theorem 10.4.1}]\label{prop:str-approx}
Let $V$ be a set of places, and suppose that there is some place $v_0$ such that $v_0 \not\in V$.
Then the diagonal inclusion of $K$ into the restricted product $\prod_{v \in V} (K_v, O_v)$ is dense.
\end{proposition}

The next proposition is the key to our proof.

\begin{proposition}\label{prop:density-for-dir-prod}
 There is an isomorphism 
 \[ %begin{align*}
      K/R_S \to \bigoplus_{v \not\in S} K_v / O_v, \quad
      [a]_{K/R_S} \mapsto ([a]_{K_v / O_v})_{v \not \in S}.
 \] %end{align*}
\end{proposition}

\begin{proof}
Let us first check that this is  well-defined.
For any given $a \in K$, we have $\absv{a}_v \le 1$, hence $[a]_{K_v / O_v} = 0$, except for finitely many places.
This means that $([a]_{K_v / O_v})_{v \not \in S}$ indeed defines an element of $\bigoplus_{v \not\in S} K_v / O_v$.

We then claim that this is a surjective homomorphism.
Take an element $([x_v]_{K_v / O_v})_v \in \bigoplus_{v \not\in S} K_v / O_v$, so that $x_v \in O_v$ except for finitely many $v$.
In particular, $(x_v)_v$ is an element in the restricted product $\prod_{v \not\in S} (K_v, O_v)$.
By Proposition \ref{prop:str-approx}, the open neighborhood $K + \prod_v O_v$ of $K$ in $\prod_{v \not\in S} (K_v, O_v)$ contains $(x_v)_v$.
This implies that $([x_v]_{K_v / O_v})_v$ is in the image of $K$.

Finally, we check that the above homomorphism is injective.
Take an element $a \in K$ that goes to $0 \in \bigoplus_{v \not\in S} K_v / O_v$.
This means that $\absv{a}_v \le 1$ for $v \not\in S$, which is the defining condition for $R_S$, hence $[a]_{K / R_S} = 0$.
\end{proof}

\begin{proposition}
The inclusion $\bA_{K, S} \to \bA_K$ induces an isomorphism $\bA_{K,S} / R_S \cong \bA_K / K$.
\end{proposition}

\begin{proof}
Since $R_S = K \cap \bA_{K, S}$, the induced map $\bA_{K,S} / R_S \to \bA_K / K$ is injective.
Let us check that it is also surjective, or equivalently, $\bA_K = K + \bA_{K, S}$.
Observe that $\bA_K / \bA_{K, S}$ is isomorphic to $\bigoplus_{v \not\in S} K_v / O_v$.
Then Proposition \ref{prop:density-for-dir-prod} implies the claim.
\end{proof}

\begin{corollary}
There is an exact sequence
\begin{equation}\label{eq:ex-seq-cpt-grps}
0 \to \prod_{v \not\in S} O_v \to \hat K \to G_S \to 0.
\end{equation}
\end{corollary}

\begin{proof}[Proof of Theorem \ref{thm:dual-of-R-S}]
Taking the dual groups of the terms in the exact sequence \eqref{eq:ex-seq-cpt-grps}, it is enough to check that the dual of $\prod_{v \not\in S} O_v$ is isomorphic to $K / R_S$.
By Corollary \ref{cor:dual-of-O-v}, this dual is isomorphic to $\bigoplus_{v \not\in S} K_v / O_v$.
Again Proposition \ref{prop:density-for-dir-prod} implies the claim.
\end{proof}

% *****************************************************************
% Ending
%******************************************************************

\raggedright


\begin{bibdiv}
\begin{biblist}

\bib{MR3722566}{article}{
      author={Amini, Massoud},
      author={Putnam, Ian~F.},
      author={Saeidi~Gholikandi, Sarah},
       title={Homology for one-dimensional solenoids},
        date={2017},
        ISSN={0025-5521},
     journal={Math. Scand.},
      volume={121},
      number={2},
       pages={219\ndash 242},
         url={https://doi.org/10.7146/math.scand.a-26265},
}

\bib{atiyah:vecell}{article}{
 Author = {Atiyah, Michael F.},
 Title = {Vector bundles over an elliptic curve},
 FJournal = {Proceedings of the London Mathematical Society. Third Series},
 Journal = {Proc. Lond. Math. Soc. (3)},
 ISSN = {0024-6115},
 number = {7},
 Pages = {414--452},
 Year = {1957},
 Language = {English},
 DOI = {10.1112/plms/s3-7.1.414},
}

\bib{black:kth}{book}{
    AUTHOR = {Blackadar, Bruce},
     TITLE = {{$K$}-theory for operator algebras},
    SERIES = {Mathematical Sciences Research Institute Publications},
    VOLUME = {5},
 PUBLISHER = {Cambridge University Press, Cambridge},
      year = {1998},
     PAGES = {xx+300},
      ISBN = {0-521-63532-2},
}

\bib{bdgw:matui}{article}{
 Author = {B{\"o}nicke, Christian},
 Author ={Dell'Aiera, Cl{\'e}ment},
 Author ={Gabe, James},
 Author ={Willett, Rufus},
 Title = {Dynamic asymptotic dimension and {Matui}'s {HK} conjecture},
 FJournal = {Proceedings of the London Mathematical Society. Third Series},
 Journal = {Proc. Lond. Math. Soc.},
 ISSN = {0024-6115},
 Volume = {126},
 Number = {4},
 Pages = {1182--1253},
 Year = {2023},
 Language = {English},
 DOI = {10.1112/plms.12510},
}

\bib{val:kthpgrp}{article}{
      author={B\"{o}nicke, Christian},
      author={Proietti, Valerio},
       title={Categorical approach to the {B}aum--{C}onnes conjecture for
  \'{e}tale groupoids},
      journal = {Journal of the Institute of Mathematics of Jussieu},
        date={2024-01-02},
         url={https://doi.org/10.1017/S1474748023000531},
         doi={10.1017/S1474748023000531},
	pages={1--46},
	note= {Published Online},
}

\bib{MR1963683}{book}{
      author={Brin, Michael},
      author={Stuck, Garrett},
       title={Introduction to dynamical systems},
   publisher={Cambridge University Press, Cambridge},
        date={2002},
        ISBN={0-521-80841-3},
         url={https://doi.org/10.1017/CBO9780511755316},
         doi={10.1017/CBO9780511755316},
}

\bib{MR1324339}{book}{
      author={Brown, Kenneth~S.},
       title={Cohomology of groups},
      series={Graduate Texts in Mathematics},
   publisher={Springer-Verlag},
     address={New York},
        date={1994},
      volume={87},
        ISBN={0-387-90688-6},
        note={Corrected reprint of the 1982 original},
}

\bib{buput:ntsol}{article}{
    AUTHOR = {Burke, Nigel D.},
    author ={Putnam, Ian F.},
     TITLE = {Markov partitions and homology for {$n/m$}-solenoids},
   JOURNAL = {Ergodic Theory Dynam. Systems},
  FJOURNAL = {Ergodic Theory and Dynamical Systems},
    VOLUME = {37},
    NUMBER = {3},
     YEAR = {2017},
     PAGES = {716--738},
      ISSN = {0143-3857,1469-4417},
       DOI = {10.1017/etds.2015.71},
       URL = {https://doi.org/10.1017/etds.2015.71},
}

\bib{MR0861410}{book}{
      author={Cassels, J. W.~S.},
       title={Local fields},
      series={London Mathematical Society Student Texts},
   publisher={Cambridge University Press, Cambridge},
        date={1986},
      volume={3},
        ISBN={0-521-30484-9; 0-521-31525-5},
         url={https://doi.org/10.1017/CBO9781139171885},
         doi={10.1017/CBO9781139171885},
}

\bib{deemany:flat}{article}{
      author={Chaiser, Rachel},
      author={Coates-Welsh, Maeve},
      author={Deeley, Robin~J.},
      author={Farhner, Annika},
      author={Giornozi, Jamal},
      author={Huq, Robi},
      author={Lorenzo, Levi},
      author={Oyola-Cortes, Jos\'{e}},
      author={Reardon, Maggie},
      author={Stocker, Andrew~M.},
       title={Invariants for the {S}male space associated to an expanding
  endomorphism of a flat manifold},
        date={2023},
        ISSN={1867-5778,1867-5786},
     journal={M\"{u}nster J. Math.},
      volume={16},
      number={1},
       pages={177\ndash 199},
}

\bib{MR605351}{article}{
      author={Connes, Alain},
       title={An analogue of the {T}hom isomorphism for crossed products of a
  {$C\sp{\ast} $}-algebra by an action of {${\bf R}$}},
        date={1981},
        ISSN={0001-8708},
     journal={Adv. in Math.},
      volume={39},
      number={1},
       pages={31\ndash 55},
}

\bib{cramo:hom}{article}{
      author={Crainic, Marius},
      author={Moerdijk, Ieke},
       title={A homology theory for {\'e}tale groupoids},
        date={2000},
        ISSN={0075-4102},
     journal={J. Reine Angew. Math.},
      volume={521},
       pages={25\ndash 46},
         url={http://dx.doi.org/10.1515/crll.2000.029},
         doi={10.1515/crll.2000.029},
}

\bib{deeley:hk}{article}{
      author={Deeley, Robin~J.},
       title={A counterexample to the {HK}-conjecture that is principal},
        date={2023},
        ISSN={0143-3857},
     journal={Ergodic Theory Dynam. Systems},
      volume={43},
      number={6},
       pages={1829\ndash 1846},
         url={https://doi.org/10.1017/etds.2022.25},
         doi={10.1017/etds.2022.25},
}

\bib{dkw:dyn}{article}{
      author={Deeley, Robin~J.},
      author={Killough, D.~Brady},
      author={Whittaker, Michael~F.},
       title={Dynamical correspondences for {S}male spaces},
        date={2016},
        ISSN={1076-9803},
     journal={New York J. Math.},
      volume={22},
       pages={943\ndash 988},
         url={http://nyjm.albany.edu:8000/j/2016/22_943.html},
}

\bib{MR2808264}{article}{
      author={Dekimpe, Karel},
       title={What is {$\ldots$}\, an infra-nilmanifold endomorphism?},
        date={2011},
        ISSN={0002-9920},
     journal={Notices Amer. Math. Soc.},
      volume={58},
      number={5},
       pages={688\ndash 689},
}

\bib{MR0482697}{book}{
      author={Engelking, Ryszard},
       title={Dimension theory},
      series={North-Holland Mathematical Library},
   publisher={North-Holland Publishing Co., Amsterdam-Oxford-New York;
  PWN---Polish Scientific Publishers, Warsaw},
        date={1978},
      volume={19},
        ISBN={0-444-85176-3},
         url={https://mathscinet.ams.org/mathscinet-getitem?mr=0482697},
        note={Translated from the Polish and revised by the author},
}

\bib{ejlo:balian}{article}{
 Author = {Enstad, Ulrik },
 Author = {Jakobsen, Mads S.},
 Author = {Luef, Franz},
 Author = {Omland, Tron},
 Title = {Deformations and {Balian}-{Low} theorems for {Gabor} frames on the adeles},
 Journal = {Advances in Mathematics},
 ISSN = {0001-8708},
 Volume = {410 B},
 Pages = {46},
 Note = {Id/No 108771},
 Year = {2022},
 Language = {English},
 DOI = {10.1016/j.aim.2022.108771},
}

\bib{skandalis:fack}{article}{
      author={Fack, Thierry},
      author={Skandalis, Georges},
       title={Connes' analogue of the {T}hom isomorphism for the {K}asparov
  groups},
        date={1981},
        ISSN={0020-9910},
     journal={Invent. Math.},
      volume={64},
      number={1},
       pages={7\ndash 14},
         url={http://dx.doi.org/10.1007/BF01393931},
         doi={10.1007/BF01393931},
}

\bib{fkps:hk}{article}{
 Author = {Farsi, Carla},
 Author ={Kumjian, Alex},
 Author ={Pask, David},
 Author ={Sims, Aidan},
 Title = {Ample groupoids: equivalence, homology, and {Matui}'s {HK} conjecture},
 FJournal = {M{\"u}nster Journal of Mathematics},
 Journal = {M{\"u}nster J. Math.},
 ISSN = {1867-5778},
 Volume = {12},
 Number = {2},
 Pages = {411--451},
 Year = {2019},
 Language = {English},
 DOI = {10.17879/53149724091},
}

\bib{MR1950475}{book}{
      author={Gelfand, Sergei~I.},
      author={Manin, Yuri~I.},
       title={Methods of homological algebra},
     edition={Second},
      series={Springer Monographs in Mathematics},
   publisher={Springer-Verlag, Berlin},
        date={2003},
        ISBN={3-540-43583-2},
         url={https://doi.org/10.1007/978-3-662-12492-5},
         doi={10.1007/978-3-662-12492-5},
      review={\MR{1950475}},
}

\bib{MR0345092}{book}{
      author={Godement, Roger},
       title={Topologie alg\'{e}brique et th\'{e}orie des faisceaux},
   publisher={Hermann, Paris},
        date={1973},
        note={Troisi\`eme \'{e}dition revue et corrig\'{e}e, Publications de
  l'Institut de Math\'{e}matique de l'Universit\'{e} de Strasbourg, XIII,
  Actualit\'{e}s Scientifiques et Industrielles, No. 1252},
}

\bib{goja:simp}{book}{
      author={Goerss, Paul~G.},
      author={Jardine, John~F.},
       title={Simplicial homotopy theory},
      series={Modern Birkh\"auser Classics},
   publisher={Birkh\"auser Verlag, Basel},
        date={2009},
        ISBN={978-3-0346-0188-7},
         url={https://doi.org/10.1007/978-3-0346-0189-4},
        note={Reprint of the 1999 edition},
}

\bib{hatcher:alg}{book}{
      author={Hatcher, Allen},
       title={Algebraic topology},
   publisher={Cambridge University Press, Cambridge},
        date={2002},
        ISBN={0-521-79160-X; 0-521-79540-0},
}

\bib{higkas:bc}{article}{
      author={Higson, Nigel},
      author={Kasparov, Gennadi},
       title={{$E$}-theory and {$KK$}-theory for groups which act properly and
  isometrically on {H}ilbert space},
        date={2001},
        ISSN={0020-9910},
     journal={Invent. Math.},
      volume={144},
      number={1},
       pages={23\ndash 74},
         url={http://dx.doi.org/10.1007/s002220000118},
         doi={10.1007/s002220000118},
}

\bib{MR0271991}{incollection}{
      author={Hirsch, Morris~W.},
      author={Pugh, Charles~C.},
       title={Stable manifolds and hyperbolic sets},
        date={1970},
   booktitle={Global {A}nalysis ({P}roc. {S}ympos. {P}ure {M}ath., {V}ols.
  {XIV}, {XV}, {XVI}, {B}erkeley, {C}alif., 1968)},
      series={Proc. Sympos. Pure Math.},
      volume={XIV-XVI},
   publisher={Amer. Math. Soc., Providence, RI},
       pages={133\ndash 163},
      review={\MR{271991}},
}

\bib{MR1249482}{book}{
      author={Husemoller, Dale},
       title={Fibre bundles},
     edition={Third},
      series={Graduate Texts in Mathematics},
   publisher={Springer-Verlag, New York},
        date={1994},
      volume={20},
        ISBN={0-387-94087-1},
         url={https://doi.org/10.1007/978-1-4757-2261-1},
         doi={10.1007/978-1-4757-2261-1},
      review={\MR{1249482}},
}

\bib{kascha:shman}{book}{
      author={Kashiwara, Masaki},
      author={Schapira, Pierre},
       title={Sheaves on manifolds},
      series={Grundlehren der Mathematischen Wissenschaften [Fundamental
  Principles of Mathematical Sciences]},
   publisher={Springer-Verlag, Berlin},
        date={1994},
      volume={292},
        ISBN={3-540-51861-4},
        note={With a chapter in French by Christian Houzel, Corrected reprint
  of the 1990 original},
}

\bib{kas:descent}{article}{,
    author    = {Kasparov, G. G.},
    title     = {Equivariant {$KK$}-theory and the {N}ovikov conjecture},
    volume    = {91},
    pages     = {147-201},
    year      = {1988},
    journal   = {Inventiones Mathematicae},
}

\bib{MR1503108}{article}{
      author={Latimer, Claiborne~G.},
      author={MacDuffee, C.~C.},
       title={A correspondence between classes of ideals and classes of
  matrices},
        date={1933},
        ISSN={0003-486X,1939-8980},
     journal={Ann. of Math.},
      volume={34},
      number={2},
       pages={313\ndash 316},
         url={https://doi.org/10.2307/1968204},
         doi={10.2307/1968204},
      review={\MR{1503108}},
}

\bib{matui:hklon}{article}{,
 Author = {Matui, Hiroki},
 Title = {{Homology and topological full groups of \'etale groupoids on totally disconnected spaces.}},
 FJournal = {{Proceedings of the London Mathematical Society}},
 Journal = {{Proc. Lond. Math. Soc.}},
 ISSN = {0024-6115; 1460-244X/e},
 Volume = {104},
 Number = {1},
 Pages = {27--56},
 Year = {2012},
 Publisher = {John Wiley \& Sons, Chichester; London Mathematical Society, London},
 Language = {English},
}

\bib{MR3552533}{article}{
      author={Matui, Hiroki},
       title={\'{E}tale groupoids arising from products of shifts of finite
  type},
        date={2016},
        ISSN={0001-8708},
     journal={Adv. Math.},
      volume={303},
       pages={502\ndash 548},
         url={https://doi.org/10.1016/j.aim.2016.08.023},
         doi={10.1016/j.aim.2016.08.023},
}

\bib{MR3837599}{incollection}{
      author={Matui, Hiroki},
       title={Topological full groups of \'{e}tale groupoids},
        date={2017},
   booktitle={Operator algebras and applications---the {A}bel {S}ymposium
  2015},
      series={Abel Symp.},
      volume={12},
   publisher={Springer},
       pages={203\ndash 230},
}



\bib{murewi:morita}{article}{,
	author = {Muhly, Paul S. and Renault, Jean N. and Williams, Dana P.},
	fjournal = {Journal of Operator Theory},
	issn = {0379-4024},
	journal = {J. Operator Theory},
	mrclass = {46L55 (22D25)},
	mrnumber = {873460},
	mrreviewer = {Jonathan M. Rosenberg},
	number = {1},
	pages = {3--22},
	title = {Equivalence and isomorphism for groupoid {$C^\ast$}-algebras},
	volume = {17},
	year = {1987}}

\bib{MR2162164}{book}{
      author={Nekrashevych, Volodymyr},
       title={Self-similar groups},
      series={Mathematical Surveys and Monographs},
   publisher={American Mathematical Society, Providence, RI},
        date={2005},
      volume={117},
        ISBN={0-8218-3831-8},
         url={https://doi.org/10.1090/surv/117},
         doi={10.1090/surv/117},
      review={\MR{2162164}},
}

\bib{MR2526786}{article}{
      author={Nekrashevych, Volodymyr},
       title={{$C^*$}-algebras and self-similar groups},
        date={2009},
        ISSN={0075-4102},
     journal={J. Reine Angew. Math.},
      volume={630},
       pages={59\ndash 123},
         url={https://mathscinet.ams.org/mathscinet-getitem?mr=2526786},
         doi={10.1515/CRELLE.2009.035},
      review={\MR{2526786}},
}

\bib{val:shi}{article}{
      author={Nishikawa, Shintaro},
      author={Proietti, Valerio},
       title={Groups with {S}panier-{W}hitehead duality},
        date={2020},
        ISSN={2379-1683,2379-1691},
     journal={Ann. K-Theory},
      volume={5},
      number={3},
       pages={465\ndash 500},
         url={https://doi.org/10.2140/akt.2020.5.465},
         doi={10.2140/akt.2020.5.465},
      review={\MR{4132744}},
}

\bib{val:smale}{article}{
      author={Proietti, Valerio},
       title={A note on homology for {S}male spaces},
        date={2020},
        ISSN={1661-7207},
     journal={Groups Geom. Dyn.},
      volume={14},
      number={3},
       pages={813\ndash 836},
         url={https://doi-org.ezproxy.uio.no/10.4171/ggd/564},
         doi={10.4171/ggd/564},
      review={\MR{4167022}},
}

\bib{valmak:groupoid}{article}{
      author={Proietti, V.},
      author={Yamashita, M.},
      title={{Homology and K-theory of dynamical systems I. Torsion-free ample
  groupoids}},
      journal={Ergodic Theory and Dynamical Systems},
      year={2022},
      pages={2630-2660},
      volume={42},
      number={8},
      doi={10.1017/etds.2021.50},

      label={PY22a}
}

\bib{valmak:groupoidtwo}{article}{
      author={Proietti, Valerio},
      author={Yamashita, Makoto},
       title={Homology and {$K$}-theory of dynamical systems {II}. {S}male
  spaces with totally disconnected transversal},
        date={2023},
        ISSN={1661-6952,1661-6960},
     journal={J. Noncommut. Geom.},
      volume={17},
      number={3},
       pages={957\ndash 998},

         url={https://doi.org/10.4171/jncg/494},
         doi={10.4171/jncg/494},
      review={\MR{4626307}},
}

\bib{valmak:threesmale}{misc}{
      author={Proietti, V.},
      author={Yamashita, M.},
      title={{Homology and K-theory of dynamical systems. III. Beyond stably disconnected Smale spaces}},
      year={2022},
      eprint={\href{https://arxiv.org/abs/2207.03118}{\texttt{arXiv:2207.03118 [math.KT]}}},
      label={PY22b}
}

\bib{put:algSmale}{article}{
      author={Putnam, Ian~F.},
       title={{$C^*$}-algebras from {S}male spaces},
        date={1996},
        ISSN={0008-414X},
     journal={Canad. J. Math.},
      volume={48},
      number={1},
       pages={175\ndash 195},
         url={https://doi.org/10.4153/CJM-1996-008-2},
         doi={10.4153/CJM-1996-008-2},
      review={\MR{1382481}},
}

\bib{MR1794291}{article}{
      author={Putnam, Ian~F.},
       title={Functoriality of the {$C^*$}-algebras associated with hyperbolic
  dynamical systems},
        date={2000},
        ISSN={0024-6107},
     journal={J. London Math. Soc.},
      volume={62},
      number={3},
       pages={873\ndash 884},
         url={https://doi.org/10.1112/S002461070000140X},
         doi={10.1112/S002461070000140X},
      review={\MR{1794291}},
}

\bib{put:HoSmale}{article}{
      author={Putnam, Ian~F.},
       title={A homology theory for {S}male spaces},
        date={2014},
        ISSN={0065-9266},
     journal={Mem. Amer. Math. Soc.},
      volume={232},
      number={1094},
       pages={viii+122},
         url={https://doi.org/10.1090/memo/1094},
      review={\MR{3243636}},
}

\bib{put:spiel}{article}{
      author={Putnam, Ian~F.},
      author={Spielberg, Jack},
       title={The structure of {$C^*$}-algebras associated with hyperbolic
  dynamical systems},
        date={1999},
        ISSN={0022-1236},
     journal={Journal of Functional Analysis},
      volume={163},
      number={2},
       pages={279 \ndash  299},
  url={http://www.sciencedirect.com/science/article/pii/S0022123698933791},
         doi={https://doi.org/10.1006/jfan.1998.3379},
}

\bib{MR584266}{book}{
      author={Renault, Jean},
       title={A groupoid approach to {$C^{\ast} $}-algebras},
      series={Lecture Notes in Mathematics},
   publisher={Springer, Berlin},
        date={1980},
      volume={793},
        ISBN={3-540-09977-8},
      review={\MR{584266}},
}

\bib{ruelle:thermo}{book}{
      author={Ruelle, David},
       title={Thermodynamic formalism},
     edition={Second},
      series={Cambridge Mathematical Library},
   publisher={Cambridge University Press, Cambridge},
        date={2004},
        ISBN={0-521-54649-4},
         url={https://doi.org/10.1017/CBO9780511617546},
         doi={10.1017/CBO9780511617546},
        note={The mathematical structures of equilibrium statistical
  mechanics},
      review={\MR{2129258}},
}

\bib{scarparo:hk}{article}{
      author={Scarparo, Eduardo},
       title={Homology of odometers},
        date={2020},
        ISSN={0143-3857},
     journal={Ergodic Theory Dynam. Systems},
      volume={40},
      number={9},
       pages={2541\ndash 2551},
         url={https://doi.org/10.1017/etds.2019.13},
         doi={10.1017/etds.2019.13},
      review={\MR{4130816}},
}

\bib{MR1345152}{book}{
	     label         = {Schm95},
	  author={Schmidt, Klaus},
	   title={Dynamical systems of algebraic origin},
	  series={Progress in Mathematics},
   publisher={Birkh\"{a}user Verlag, Basel},
		date={1995},
	  volume={128},
		ISBN={3-7643-5174-8},
		 url={https://mathscinet.ams.org/mathscinet-getitem?mr=1345152},
	  review={\MR{1345152}},
}

\bib{MR650021}{article}{
      label         = {Scho82},
      author={Schochet, Claude},
       title={Topological methods for {$C^{\ast} $}-algebras. {II}. {G}eometric
  resolutions and the {K}\"unneth formula},
        date={1982},
        ISSN={0030-8730},
     journal={Pacific J. Math.},
      volume={98},
      number={2},
       pages={443\ndash 458},
         url={http://projecteuclid.org/euclid.pjm/1102734267},
      review={\MR{650021}},
}

\bib{takai:ano}{incollection}{
      author={Takai, H.},
       title={{$KK$}-theory for the {$C^\ast$}-algebras of {A}nosov
  foliations},
        date={1986},
   booktitle={Geometric methods in operator algebras ({K}yoto, 1983)},
      series={Pitman Res. Notes Math. Ser.},
      volume={123},
   publisher={Longman Sci. Tech., Harlow},
       pages={387\ndash 399},
      review={\MR{866509}},
}

\bib{MR30491}{article}{
      author={Taussky, Olga},
       title={On a theorem of {L}atimer and {M}ac{D}uffee},
        date={1949},
        ISSN={0008-414X},
     journal={Canad. J. Math.},
      volume={1},
       pages={300\ndash 302},
         url={https://doi.org/10.4153/cjm-1949-026-1},
         doi={10.4153/cjm-1949-026-1},
      review={\MR{30491}},
}

\bib{tu:moy}{article}{
      author={Tu, Jean-Louis},
       title={La conjecture de {B}aum-{C}onnes pour les feuilletages
  moyennables},
        date={1999},
        ISSN={0920-3036},
     journal={$K$-Theory},
      volume={17},
      number={3},
       pages={215\ndash 264},
         url={http://dx.doi.org/10.1023/A:1007744304422},
         doi={10.1023/A:1007744304422},
      review={\MR{1703305 (2000g:19004)}},
}

\bib{MR735415}{article}{
      author={Wallace, D.~I.},
       title={Conjugacy classes of hyperbolic matrices in {${\rm Sl}(n,\,{\bf
  Z})$} and ideal classes in an order},
        date={1984},
        ISSN={0002-9947},
     journal={Trans. Amer. Math. Soc.},
      volume={283},
      number={1},
       pages={177\ndash 184},
         url={https://doi.org/10.2307/1999996},
         doi={10.2307/1999996},
      review={\MR{735415}},
}

\bib{weil:bnt}{book}{
 Author = {Weil, Andr{\'e}},
 Title = {Basic number theory},
 Series = {Grundlehren Math. Wiss.},
 ISSN = {0072-7830},
 Volume = {144},
 Year = {1974},
 Publisher = {Springer, Cham},
 Language = {English},
}

\end{biblist}
\end{bibdiv}
\end{document}